\documentclass[12pt]{amsart}
\usepackage{amsmath,amsthm,amssymb,amscd,cite,array}
\usepackage[dvips]{graphicx}
\usepackage{hyperref}

\newtheorem{pr}{Proposition}
\newtheorem{lem}[pr]{Lemma}
\newtheorem{thm}[pr]{Theorem}
\newtheorem{s}[pr]{Corollary}
\theoremstyle{remark}
\newtheorem{zam}{Remark}
\newtheorem*{obozn}{{\rm\bf Notation}}
\newtheorem*{obozns}{{\rm\bf Notation}}

\renewcommand\div{\text{ }\vdots\text{ }}
\newcommand\ndiv{\not\vdots\text{ }}
\newcommand{\myChar}{\mathrm{char\,}K}
\newcommand{\myNod}{\text{{\rm gcd}}}
\newcommand{\Hom}{\mathrm{Hom}}
\renewcommand{\Im}{\mathrm{Im}}
\newcommand{\Ker}{\mathrm{Ker}}
\newcommand{\HH}{\mathrm{HH}}
\newcommand{\N}{\mathbb{N}}
\newcommand{\Z}{\mathbb{Z}}
\newcommand{\cl}{\mathrm{cl}}
\newcommand\two[2]{\genfrac{}{}{0pt}{}{#1}{#2}}
\newcommand\three[3]{\two{#1}{\two{#2}{#3}}}
\newcommand\quatro[4]{\left(\two{\two{#1}{#2}}{\two{#3}{#4}}\right)}
\newcommand\quintet[5]{\left(\two{\two{#1}{#2}}{\three{#3}{#4}{#5}}\right)}
\newcommand\triplet[3]{\left(\genfrac{}{}{0pt}{}{#1}{\genfrac{}{}{0pt}{}{#2}{#3}}\right)}
\DeclareMathOperator{\ord}{ord}

\def\a{\alpha}
\def\b{\beta}
\def\g{\gamma}
\def\le{\leqslant}
\def\ge{\geqslant}
\def\ra{\rightarrow}

\begin{document}

\title{Hochschild cohomology rings for self-injective algebras of tree classes $E_7$ and $E_8$.}
\author{Mariya Kachalova}
\email{mashakachalova@mail.ru}

\begin{abstract}
The paper describes in terms of generators and relations the Hochschild cohomology rings 
for a self-injective algebras of tree classes $E_7$ and $E_8$ of finite representation type. 
\end{abstract}
\maketitle

\tableofcontents

\section{Introduction}

The present paper continues a series of papers devoted to studying Hochschild
cohomologies of self-injective algebra of finite representation type over an 
algebraically closed field. According to Riedtmann's classification, the stable 
$AR$-quiver of such an algebra can be described with the help of an associated 
tree which coincides with one of the Dynkin diagrams $A_n, D_n, E_6, E_7$, or
$E_8$ (see \cite{Riedt}). The complete description of the Hochschild cohomology 
ring was obtained for algebras of types $A_n$ (see \cite{Erd,Gen&Ka,Ka,Pu}), 
$D_n$ (see \cite{Volkov1,Volkov2,Volkov3,Volkov4,Volkov5,Volkov6}) and
$E_6$ (see \cite{Pu2,Ka2}). In the present paper we 
study the Hochschild cohomology rings for algebras of types $E_7$ and $E_8$. 

Any such algebra is derived equivalent to path algebra of some quiver factorized by ideal.

For type $E_7$, let $\mathcal Q_s$ ($s\in\N$) be the following quiver:
\newpage

\begin{figure}[h]
\includegraphics[width=12cm, scale=1]{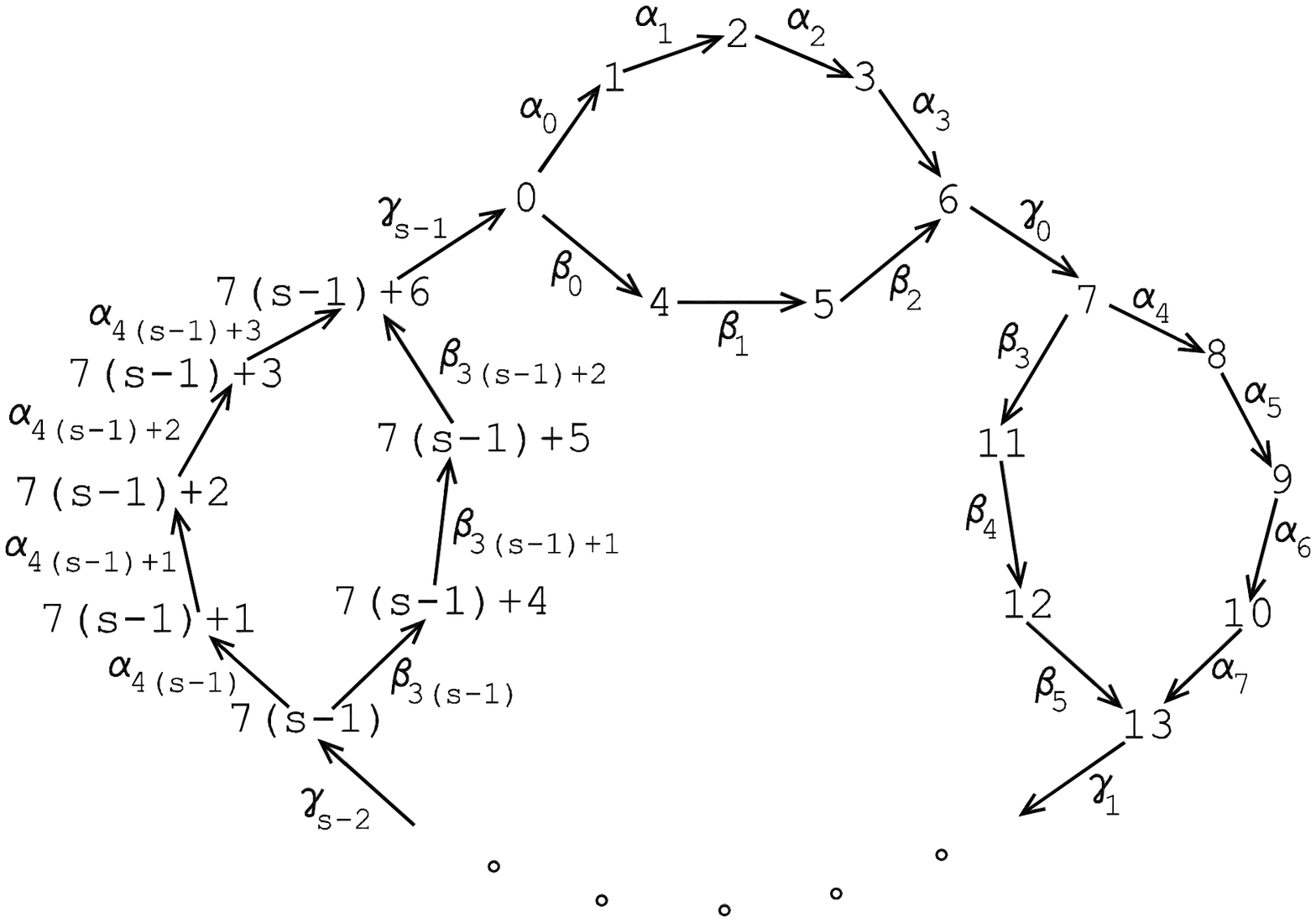}
\end{figure}

Then any algebra of the class $E_7$ is derived equivalent to the algebra 
$R_s=K\left[\mathcal Q_s\right]/I$, where $K$ is a field, and $I$ is the ideal in the path
algebra $K\left[\mathcal Q_s\right]$ of the quiver $\mathcal Q_s$, generated by

a) all the paths of length $6$;

b) the expressions of the form $\a^4-\b^3$, $\a\g\b$, $\b\g\a$, $\b^{i}\g\b^{4-i}$ ($1\le i\le 3$).

For type $E_8$, let $\mathcal Q^\prime_s$ ($s\in\N$) be the following quiver:

\begin{figure}[h]
\includegraphics[width=12cm, scale=1]{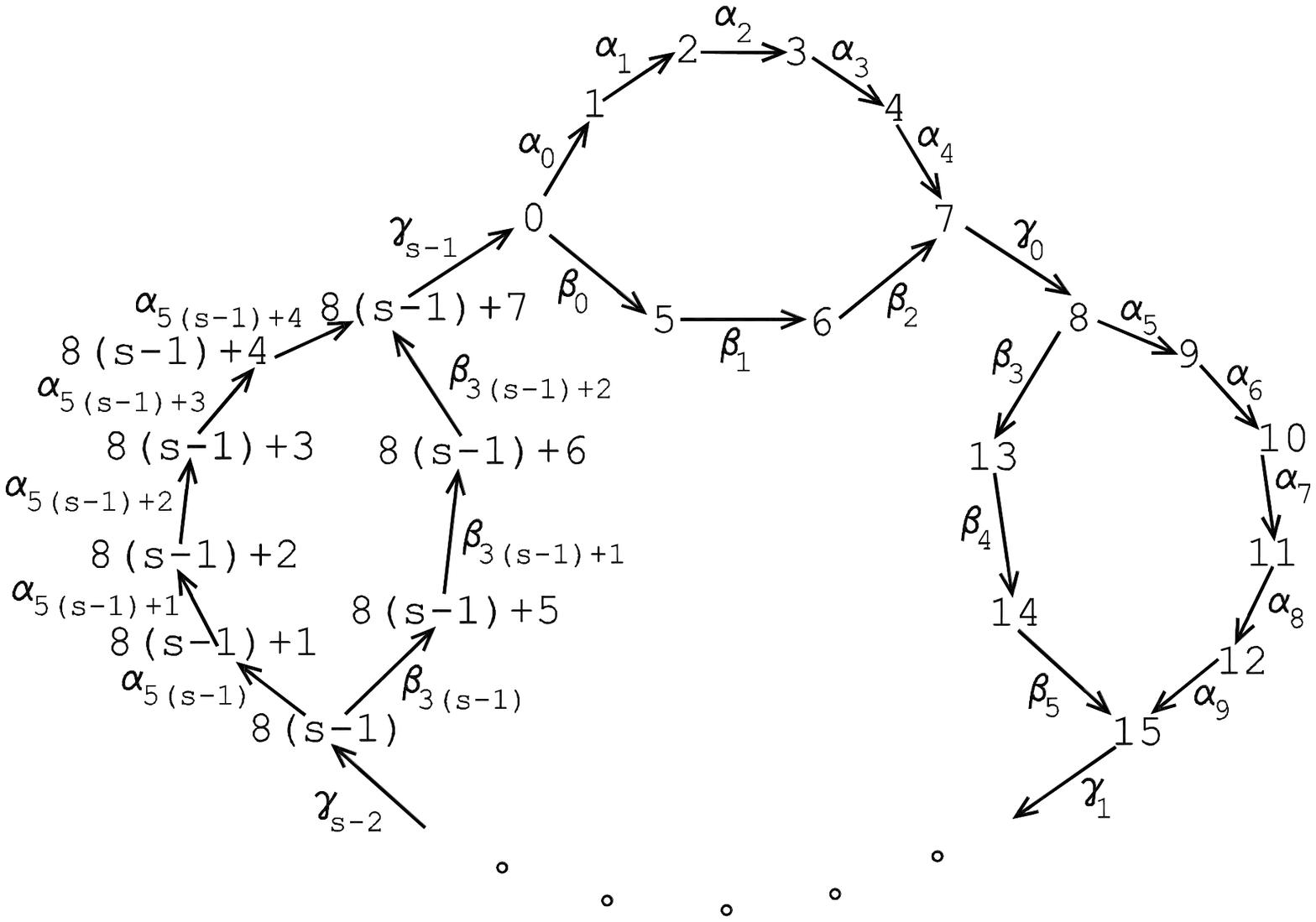}
\end{figure}

Then every algebra of type $E_8$ is derived equivalent to the algebra 
$R^\prime_s=K\left[\mathcal Q^\prime_s\right]/I^\prime$, where $K$ is a field, and $I^\prime$ is the ideal in the path
algebra $K\left[\mathcal Q^\prime_s\right]$ of the quiver $\mathcal Q^\prime_s$, generated by

a) all the paths of length $7$;

b) the expressions of the form $\a^5-\b^3$, $\a\g\b$, $\b\g\a$, $\b^{i}\g\b^{4-i}$ ($1\le i\le 3$).

\begin{zam}
We often omit indices in arrows $\a_i$, $\b_i$, and $\g_i$ as long as the subscripts are clear
from the context.
\end{zam}

The present paper is devoted to studying the Hochschild cohomology ring 
structure for the algebras $R_s$ and $R^\prime_s$. We obtain the descriptions of the Hochschild cohomology 
rings structure for this algebras in terms of generators and relations. 
We note, that to describe the cohomology rings structure, 
a bimodule resolution of $R_s$ and $R^\prime_s$ is constructed, which is interesting in itself.

\begin{zam}
We write a program that helps to find differentials, $\Omega$-shifts and other results for the present paper.
It's open source and available at \href{https://github.com/pigmasha/e8}{https://github.com/pigmasha/e8}
\end{zam}

\section{Tree class $E_7$: Statement of the main results}

In what follows, we assume $n=7$.

Let $\HH^t(R)$ be the $t$th group of the Hochschild cohomology ring of $R$ with coefficients in
$R$. Let $\ell$ be the aliquot, and $r$ be the residue of division of $t$ by $17$, $m$ be the
aliquot of division of $r$ by $2$.

Consider the case of $s>1$. To describe Hochschild cohomology ring of algebra $R_s$ we must
introduce the following conditions on an arbitrary degree $t$:

$($1$)$ $r=0$, $m+9\ell\equiv 0(s),\text{ }\ell\div 2\text{ or }\myChar=2$;\label{degs}

$($2$)$ $r=0$, $m+9\ell\equiv 1(s),\text{ }\ell\ndiv 2\text{ or }\myChar=2$;

$($3$)$ $r=1$, $m+9\ell\equiv 0(s),\text{ }\ell\div 2\text{ or }\myChar=2$;

$($4$)$ $r=3$, $m+9\ell\equiv 0(s),\text{ }\ell\ndiv 2\text{ or }\myChar=2$;

$($5$)$ $r=4$, $m+9\ell\equiv 1(s),\text{ }\ell\ndiv 2\text{ or }\myChar=2$;

$($6$)$ $r=5$, $m+9\ell\equiv 0(s),\text{ }\ell\div 2,\text{ }\myChar=3$;

$($7$)$ $r=6$, $m+9\ell\equiv 1(s),\text{ }\ell\div 2,\text{ }\myChar=3$;

$($8$)$ $r=7$, $m+9\ell\equiv 0(s),\text{ }\ell\ndiv 2\text{ or }\myChar=2$;

$($9$)$ $r=8$, $m+9\ell\equiv 0(s),\text{ }\ell\div 2\text{ or }\myChar=2$;

$($10$)$ $r=8$, $m+9\ell\equiv 1(s),\text{ }\ell\ndiv 2\text{ or }\myChar=2$;

$($11$)$ $r=9$, $m+9\ell\equiv 0(s),\text{ }\ell\div 2\text{ or }\myChar=2$;

$($12$)$ $r=10$, $m+9\ell\equiv 0(s),\text{ }\ell\ndiv 2,\text{ }\myChar=3$;

$($13$)$ $r=11$, $m+9\ell\equiv 0(s),\text{ }\ell\ndiv 2,\text{ }\myChar=3$;

$($14$)$ $r=12$, $m+9\ell\equiv 0(s),\text{ }\ell\div 2\text{ or }\myChar=2$;

$($15$)$ $r=13$, $m+9\ell\equiv 0(s),\text{ }\ell\div 2\text{ or }\myChar=2$;

$($16$)$ $r=15$, $m+9\ell\equiv 0(s),\text{ }\ell\ndiv 2\text{ or }\myChar=2$;

$($17$)$ $r=16$, $m+9\ell\equiv 0(s),\text{ }\ell\div 2\text{ or }\myChar=2$;

$($18$)$ $r=16$, $m+9\ell\equiv 1(s),\text{ }\ell\ndiv 2\text{ or }\myChar=2$.

Let $$M_0=\frac{s}{\myNod(s,9)},\quad M=\begin{cases}17M_0,\quad\myChar=2\text{ or }M_0\div 2;\\34M_0\quad\text{otherwise.}\end{cases}$$

\begin{zam}
We will prove in paragraph \ref{sect_res} that the minimal period of bimodule resolution of $R_s$
is $M$.
\end{zam}

Let $\{t_{1, i},\dots,t_{\alpha_i, i} \}$ be a set of all degrees $t$, that satisfy the conditions
of item $i$ from the above list, and such that $0\le t_{j, i}<M$ $(j=1,\dots,\alpha_i)$. Consider
the set $$\mathcal
X=\bigcup_{i=1}^{18}\left\{X^{(i)}_{t_{j,i}}\right\}_{j=1}^{\alpha_i}\cup\{T\},$$ and define a
graduation of polynomial ring $K[\mathcal X]$ such that
\begin{align*}\label{degs2}
&\deg X^{(i)}_{t_{j,i}}=t_{j, i} \:\text{for all} \: i=1,\dots,18 \:\text{and}\: j=1,\dots,\alpha_i;\tag{$\circ$}\\
&\deg T=M.
\end{align*}

\begin{zam}\label{brief_notation}
Hereafter we shall use simplified denotation $X^{(i)}$ instead of $X^{(i)}_{t_{j,i}}$, since lower
indexes are clear from context.
\end{zam}

\begin{obozn}
$$\widetilde X^{(i)}=\begin{cases}X^{(i)},\quad\deg\widetilde
X^{(i)}<\deg T;\\TX^{(i)},\quad\text{otherwise.}\end{cases}$$
\end{obozn}

Define a graduate $K$-algebra $\mathcal A=K[\mathcal X]/I$, where $I$ is the ideal generated by
homogeneous elements corresponding to the following relations.
\begin{align*}
&X^{(3)}X^{(2)}=X^{(3)}X^{(3)}=X^{(3)}X^{(5)}=X^{(3)}X^{(6)}=X^{(3)}X^{(7)}=X^{(3)}X^{(10)}=0;\\
&X^{(3)}X^{(11)}=X^{(3)}X^{(12)}=X^{(3)}X^{(13)}=X^{(3)}X^{(15)}=X^{(3)}X^{(18)}=0;\\
&X^{(3)}X^{(1)}=\widetilde X^{(3)},\quad X^{(3)}X^{(4)}=\widetilde X^{(5)},\quad X^{(3)}X^{(8)}=\widetilde X^{(10)},\quad X^{(3)}X^{(9)}=\widetilde X^{(11)};\\
&X^{(3)}X^{(14)}=\widetilde X^{(15)},\quad X^{(3)}X^{(16)}=\widetilde X^{(18)},\quad X^{(3)}X^{(17)}=3\widetilde X^{(2)};\\
\end{align*}
\begin{align*}
X^{(4)}X^{(4)}&=\begin{cases}-s\widetilde X^{(7)},\quad\myChar
=3,\\0,\quad\text{otherwise};\end{cases}&\text{(r1)}\\
X^{(4)}X^{(9)}&=\begin{cases}-s\widetilde X^{(13)},\quad\myChar
=3,\\0,\quad\text{otherwise};\end{cases}&\text{(r2)}\\
X^{(8)}X^{(17)}&=\begin{cases}s\widetilde X^{(7)},\quad\myChar
=3,\\0,\quad\text{otherwise};\end{cases}&\text{(r3)}\\
X^{(9)}X^{(16)}&=\begin{cases}-s\widetilde X^{(7)},\quad\myChar
=3,\\0,\quad\text{otherwise};\end{cases}&\text{(r4)}\\
X^{(14)}X^{(17)}&=\begin{cases}-s\widetilde X^{(13)},\quad\myChar
=3,\\0,\quad\text{otherwise}.\end{cases}&\text{(r5)}
\end{align*}

Describe the rest relations as a tables (numbers (r1)--(r5) in tables cells are the number of
relation that defines a multiplication of the following elements).

\setlength{\extrarowheight}{1mm}
\begin{tabular}{c|c|c|c|c|c|c}
&$X^{(1)}$&$X^{(2)}$&$X^{(4)}$&$X^{(6)}$&$X^{(7)}$&$X^{(8)}$\\
\hline
$X^{(1)}$&$X^{(1)}$&$X^{(2)}$&$X^{(4)}$&$X^{(6)}$&$X^{(7)}$&$X^{(8)}$ \\
\hline
$X^{(2)}$& &0&0&0&0&0 \\
\hline
$X^{(4)}$& & & (r1)&$-sX^{(10)}$&0&0 \\
\hline
$X^{(6)}$& & & &0&0&0\\
\hline
$X^{(7)}$& & & & &0&0 \\
\hline
$X^{(8)}$& & & & & &0\\
\end{tabular}

$\quad$

$\quad$

\begin{tabular}{c|c|c|c|c|c|c}
&$X^{(9)}$&$X^{(12)}$&$X^{(13)}$&$X^{(14)}$&$X^{(16)}$&$X^{(17)}$\\
\hline
$X^{(1)}$&$X^{(9)}$&$X^{(12)}$&$X^{(13)}$&$X^{(14)}$&$X^{(16)}$&$X^{(17)}$ \\
\hline
$X^{(2)}$&$X^{(10)}$&0&0&0&0&$-X^{(18)}$ \\
\hline
$X^{(4)}$& (r2)&$-sX^{(15)}$&0&$X^{(16)}$&0&0 \\
\hline
$X^{(6)}$&$sX^{(15)}$&0&$-X^{(18)}$&$-sX^{(2)}$&0&$sX^{(5)}$\\
\hline
$X^{(7)}$&0&$X^{(18)}$&0&0&0&0\\
\hline
$X^{(8)}$&$-X^{(16)}$&$-sX^{(2)}$&0&0&0& (r3)\\
\end{tabular}

$\quad$

$\quad$

\begin{tabular}{c|c|c|c|c|c|c}
&$X^{(9)}$&$X^{(12)}$&$X^{(13)}$&$X^{(14)}$&$X^{(16)}$&$X^{(17)}$\\
\hline
$X^{(9)}$&$X^{(17)}$&$-sX^{(3)}$&0&$X^{(4)}$&(r4)&$3X^{(8)}$ \\
\hline
$X^{(12)}$& &0&$X^{(5)}$&$-X^{(6)}$&$sX^{(10)}$&$-sX^{(11)}$ \\
\hline
$X^{(13)}$& & &0&$X^{(7)}$&0&0\\
\hline
$X^{(14)}$& & & &$-X^{(8)}$&0&(r5) \\
\hline
$X^{(16)}$& & & & &0&0 \\
\hline
$X^{(17)}$& & & & & &$-3X^{(16)}$\\
\end{tabular}

\begin{thm}\label{main_thm}
Let $s>1$, $R=R_s$ is algebra of the type $E_7$. Then the Hochschild cohomology ring $\HH^*(R)$ is
isomorphic to $\mathcal A$ as a graded $K$-algebra.
\end{thm}

Consider the case of $s=1$.

Let us introduce the set $$\mathcal X^\prime=\begin{cases}\mathcal X\cup \left\{X^{(19)}_0, X^{(20)}_0, X^{(21)}_0, 
 X^{(22)}_0, X^{(23)}_0, X^{(24)}_0, X^{(25)}_0\right\},\quad\myChar\ne 2;\\
\mathcal X\cup \left\{X^{(20)}_0, X^{(21)}_0, 
 X^{(22)}_0, X^{(23)}_0, X^{(24)}_0, X^{(25)}_0\right\}, \quad
\myChar=2;\end{cases}$$ and define a graduation of polynomial ring $K[\mathcal X^\prime]$ such that
\begin{align*}
&\deg X^{(i)}_{t_{j,i}}=t_{j, i} \:\text{for all} \: i=1,\dots,18 \:\text{and}\: j=1,\dots,\alpha_i;\\
&\deg T=M \text{ (similar to (\ref{degs2}))};\\&\deg X^{(i)}_0=0 \:\text{for all} \: i=19,\dots,25.
\end{align*}

Define a graduate $K$-algebra $\mathcal A^\prime=K[\mathcal X^\prime]/I^\prime$, where $I^\prime$
is the ideal generated by homogeneous elements corresponding to the relations described in the case
of $s>1$, and by the following relations:
\begin{align*}
X^{(1)}X^{(i)}=&\begin{cases}\widetilde X^{(i)},\quad
t_1=0;\\\widetilde X^{(2)},\quad t_1>0,\text{ }i\in[22,25],\text{
}\myChar=2;\\0,\quad\text{otherwise};\end{cases}\\
X^{(9)}X^{(i)}=&\begin{cases}\widetilde X^{(10)},\quad i\in[22,25],\text{ }
\myChar=2;\\0,\quad\text{otherwise};\end{cases}\\
X^{(17)}X^{(i)}=&\begin{cases}\widetilde X^{(18)},\quad i\in[22,25],\text{ }
\myChar=2;\\0,\quad\text{otherwise};\end{cases}\\
X^{(j)}X^{(i)}=&0,\quad j\in[2, 25]\setminus\{9,17\},\quad
i\in[19, 25],
\end{align*}
where $t_1$ denotes a degree of the element $X^{(1)}$.

\begin{thm}\label{main_thm2}
Let $s=1$, $R=R_1$ is algebra of the type $E_7$. Then the Hochschild cohomology ring $\HH^*(R)$ is
isomorphic to $\mathcal A^\prime$ as a graded $K$-algebra.
\end{thm}

\begin{zam}
From the descriptions of rings $\HH^*(R)$ given in theorems
\ref{main_thm} and \ref{main_thm2} it implies, in particular, that
they are commutative.
\end{zam}

\section{Tree class $E_7$: Bimodule resolution}\label{sect_res}

We will construct the minimal projective bimodule resolution of the $R$ in the following form: $$
\dots\longrightarrow Q_3\stackrel{d_2}\longrightarrow Q_2\stackrel{d_1}\longrightarrow
Q_1\stackrel{d_0}\longrightarrow Q_0\stackrel\varepsilon\longrightarrow R\longrightarrow 0
$$

Let $\Lambda$ be an enveloping algebra of algebra $R$. Then $R$--$R$-bimodules can be considered as
left $\Lambda$-modules.

\begin{obozns}$\quad$

(1) Let $e_i,\text{ }i\in \Z_{7s}=\{0, 1,\dots, 7s-1\},$ be the idempotents of the algebra
$K\left[\mathcal Q_s\right]$, that correspond to the vertices of the quiver $\mathcal Q_s$.

(2) Denote by $P_{i,j}=R(e_i\otimes e_j)R=\Lambda(e_i\otimes e_j)$, $i,j\in \Z_{7s}$. Note that the
modules $P_{i,j}$, forms the full set of the (pairwise non-isomorphic
 by) indecomposable projective $\Lambda$-modules.

(3) For $a\in\Z$, $t\in\N$ we denote the smallest nonnegative deduction of $a$ modulo $t$ with
$(a)_t$ (in particular, $0\le(a)_t\le t-1$).

\end{obozns}

Let $R=R_s$. We introduce an automorphism $\sigma\text{: }R\rightarrow R$, which is mapping as
follows:
$$\sigma(e_i)=e_{i+9n},\quad\sigma(\g_i)=-\g_{i+9},\quad\sigma(\a_i)=-\a_{i+9\cdot 4},$$
 $$\sigma(\b_i)=\begin{cases}
 \b_{i+9\cdot 3},\quad (i)_3=0;\\
 -\b_{i+9\cdot 3},\quad (i)_3\ne 0;\\
 \end{cases}$$

Define the helper functions $f\text{: }\Z\times\Z\rightarrow\Z$ and $f\text{: }\Z\times\Z\times\Z\rightarrow\Z$, 
 which act in the following way:
$$f(x,y)=\begin{cases}1,\quad x=y;\\0,\quad x\ne y,\end{cases}\quad
f(x,y_1,y_2)=\begin{cases}1,\quad y_1\le x\le y_2;\\0,\quad\text{otherwise.}
\end{cases}$$

Introduce $Q_r\text{ }(r\le 16)$. Let $m$ be the aliquot of division of $r$ by $2$ for considered
degree $r$. We have
\begin{align*}
Q_{2m}&=\bigoplus_{r=0}^{s-1} Q_{2m,r}^\prime,\quad 0\le m\le 8,\\
Q_{2m+1}&=\bigoplus_{r=0}^{s-1} Q_{2m+1,r}^\prime,\quad 0\le m\le 7.
\end{align*}
\begin{multline*}
Q_{2m,r}^\prime=\left(\bigoplus_{i=0}^{f(m,2,5)}P_{b_0(r,m,i),7r}\right)
\oplus\bigoplus_{j=0}^2\bigoplus_{i=0}^{f(m+j,4)+f(m+j,6)}P_{b_1(r,m,i,j),7r+j+1}\\
\oplus\bigoplus_{j=0}^1\bigoplus_{i=0}^{f(m+j,3,6)}P_{b_2(r,m,i,j),7r+j+4}
\oplus\left(\bigoplus_{i=0}^{f(m,3,6)}P_{b_3(r,m,i),7r+6}\right),
\end{multline*} 
where
$$\begin{aligned}
b_0(r,m,i)=&7(r+m)-f(i,0)(1-f(m,0)-f(m,6)-f(m,8));\\
\quad b_1(r,m,i,j)=&7(r+m)+m+1+j-4f(i,0)f(m+j,4,5)\\&-f(m+j,6)(f(i,0)+3)-3f(m+j,7)-8f(m+j,8,10);\\
b_2(r,m,i,j)=&7(r+m)+m+4+j-5f(m+j,2)-f(m+j,3,4)(2f(i,0)+3)\\&-f(m+j,5)(3f(i,1)+5)-f(m+j,6)(3f(i,0)+5)-8f(m+j,7,9);\\
b_3(r,m,i)=&7(r+m)+6+f(i,0)(f(m,1)+f(m,7))+f(i,1).
\end{aligned}$$

\begin{multline*}
Q_{2m+1,r}^\prime=\left(\bigoplus_{i=0}^{1+f(m,2,4)-f(m,7)}P_{b_4(r,m,i),7r}\right)
\oplus\bigoplus_{j=0}^2P_{b_5(r,m,j),7r+j+1}\\
\oplus\bigoplus_{j=0}^1\bigoplus_{i=0}^{f(m+j,4)}P_{b_6(r,m,i,j),7r+j+4}
\oplus\left(\bigoplus_{i=0}^{1+f(m,3,5)-f(m,0)}P_{b_7(r,m,i),7r+6}\right),
\end{multline*}
where
$$\begin{aligned}
b_4(r,m,i)=&7(r+m)+m+1+3f(i,1)f(m,0,1)+f(m,2)(f(i,0)-2f(i,2))\\
&+f(m,3)(f(i,1)-2f(i,0))-2f(m,4)(f(i,1)+2f(i,2))\\
&-2f(m,5,6)(1+f(i,0))-2f(m,7);\\
b_5(r,m,j)=&7(r+m)+m+j+2+2f(m+j,2,3)-2f(m+j,6,9);\\
b_6(r,m,i,j)=&7(r+m)+m+j+5-2f(m+j,3)-f(m+j,4)(2+f(i,0))\\
&-3f(m+j,5)-5(f(m+j,6)+f(m+j,7))-2f(m+j,8);\\
b_7(r,m,i)=&7(r+m+1)+m+3f(m,1,2)f(i,1)+f(m,3)(f(i,0)-2f(i,2))\\
&+f(m,4)(f(i,1)-2f(i,0))-2f(m,5)(f(i,1)+2f(i,2))\\
&-2f(m,6,7)(2f(i,0)+f(i,1)).
\end{aligned}$$

\begin{thm}\label{resol_thm}
Let $R=R_s$ is algebra of the type $E_7$. Then the minimal projective resolution of the
$\Lambda$-module $R$ is of the form:
\begin{equation}\label{resolv}\tag{$+$} \dots\longrightarrow
Q_3\stackrel{d_2}\longrightarrow Q_2\stackrel{d_1}\longrightarrow
Q_1\stackrel{d_0}\longrightarrow
Q_0\stackrel\varepsilon\longrightarrow R\longrightarrow
0,
\end{equation}
where $\varepsilon$ is the multiplication map $(\varepsilon(a\otimes b)=ab)$; $Q_r\text{ }(r\le
16)$ are as above, $d_r\text{ }(r\le 16)$ are described in ancillary files of this paper; further $Q_{17\ell+r}$, where $\ell\in \N$
and $0\le r\le 16$, is obtained from $Q_r$ by replacing every direct summand $P_{i,j}$ to
$P_{\sigma^\ell(i),j}$ correspondingly $($here $\sigma(i)=j$, if $\sigma(e_i)=e_j)$, and the
differential $d_{17\ell+r}$ is obtained from $d_r$ by act of $\sigma^\ell$ by all left tensor
components of the corresponding matrix.
\end{thm}

To prove that the terms $Q_i$ are of this form we introduce $P_i=Re_i$ is the projective cover of
the simple $R$-modules $S_i$, corresponding to the vertices of the quiver $\mathcal Q_s$. We will
find projective resolutions of the simple $R$-modules $S_i$.

\begin{obozn}
For $R$-module $M$ its $m$th syzygy is denoted by $\Omega^m(M)$.
\end{obozn}

\begin{zam}\label{note_brev}
From here we denote the multiplication homomorphism from the right by an element $w$ by $w$.
\end{zam}

\begin{lem}\label{lem_s0}
The begin of the minimal projective resolution of $S_{7r}$ is of the form

\begin{multline*}
\dots\longrightarrow P_{7(r+6)+6} \stackrel{\binom{\a}{-\b}}
\longrightarrow P_{7(r+6)+3}\oplus P_{7(r+6)+5}\stackrel{(\a^{3}\text{ }\b^{2})}\longrightarrow\\
\longrightarrow P_{7(r+6)}
\stackrel{\binom{\g\a^{2}}{-\g\b^{2}}}
\longrightarrow P_{7(r+5)+2}\oplus P_{7(r+5)+4}
\stackrel{\binom{\phantom{-}\a^{2}\g\phantom{-}0}{\a^{2}\phantom{-}\b}}
\longrightarrow \\
\longrightarrow P_{7(r+4)+6}\oplus P_{7(r+5)}
\stackrel{\triplet{\phantom{-}\b\phantom{-}0}{-\a\phantom{-}\g\a}{-\a^{3}\phantom{-}0}}
\longrightarrow P_{7(r+4)+5}\oplus P_{7(r+4)+3}\oplus P_{7(r+4)+1}
\stackrel{\binom{\b^{2}\g\phantom{-}\a^{3}\g\phantom{-}0}{\b^{2}\phantom{-}\phantom{-}0\phantom{-}\phantom{-}\a}}
\longrightarrow \\
\longrightarrow P_{7(r+3)+6}\oplus P_{7(r+4)}
\stackrel{\triplet{\phantom{-}\a^{2}\phantom{-}0}{-\b\phantom{-}\g\b}{-\b^{2}\phantom{-}0}}
\longrightarrow P_{7(r+3)+2}\oplus P_{7(r+3)+5}\oplus P_{7(r+3)+4}
\stackrel{\binom{\a^{2}\g\phantom{-}\b^{2}\g\phantom{-}0}{\a^{2}\phantom{-}\phantom{-}0\phantom{-}\phantom{-}\b}}
\longrightarrow \\
\longrightarrow P_{7(r+2)+6}\oplus P_{7(r+3)} 
\stackrel{\triplet{\phantom{-}\b^{2}\phantom{-}0}{-\a\phantom{-}\g\a}{-\a^{3}\phantom{-}0}}
\longrightarrow P_{7(r+2)+4}\oplus P_{7(r+2)+3}\oplus P_{7(r+2)+1} 
\stackrel{\binom{\b\g\phantom{-}\a^{3}\g\phantom{-}0}{\b\phantom{-}\phantom{-}0\phantom{-}\phantom{-}\a}}
\longrightarrow \\
\longrightarrow P_{7(r+1)+6}\oplus P_{7(r+2)}
\stackrel{\binom{\phantom{-}\a^{2}\phantom{-}0}{-\b\phantom{-}\g\b}}\longrightarrow 
P_{7(r+1)+2}\oplus P_{7(r+1)+5} \stackrel{(\a^{2}\g\text{ }\b^{2}\g)}\longrightarrow\\ 
\longrightarrow P_{7r+6}\stackrel{\binom{\a^{3}}{-\b^{2}}}\longrightarrow P_{7r+1}\oplus P_{7r+4}
\stackrel{(\a\text{ }\b)}\longrightarrow P_{7r}\longrightarrow S_{7r}\longrightarrow 0.
\end{multline*}
At that $\Omega^{15}(S_{7r})\simeq S_{7(r+7)+6}$.
\end{lem}

\begin{lem}
The begin of the minimal projective resolution of $S_{7r+1}$ is of the form $$\dots\longrightarrow
P_{7r+2} \stackrel{\a}\longrightarrow P_{7r+1}\longrightarrow S_{7r+1}\longrightarrow 0.$$ At that
$\Omega^{2}(S_{7r+1})\simeq S_{7(r+1)+2}$.
\end{lem}

\begin{lem}
The begin of the minimal projective resolution of $S_{7r+2}$ is of the form $$\dots\longrightarrow
P_{7r+3} \stackrel{\a}\longrightarrow P_{7r+2}\longrightarrow S_{7r+2}\longrightarrow 0.$$ At that
$\Omega^{2}(S_{7r+2})\simeq S_{7(r+1)+3}$.
\end{lem}

\begin{lem}
The begin of the minimal projective resolution of $S_{7r+3}$ is of the form
\begin{multline*}
\dots\longrightarrow P_{7(r+6)+1}\stackrel{\a}\longrightarrow
P_{7(r+6)}\stackrel{\g\b}\longrightarrow
P_{7(r+5)+5}\stackrel{\b^{2}\g}\longrightarrow
P_{7(r+4)+6}\stackrel{\binom{\a}{-\b^{2}}}\longrightarrow\\
\longrightarrow
P_{7(r+4)+3}\oplus P_{7(r+4)+4}\stackrel{(\a^{3}\text{ }\b)}\longrightarrow
P_{7(r+4)}\stackrel{\g\a^{2}}\longrightarrow
P_{7(r+3)+2}\stackrel{\a^{2}\g}\longrightarrow\\
\longrightarrow P_{7(r+2)+6}\stackrel{\binom{\a^{3}}{-\b}}\longrightarrow
P_{7(r+2)+1}\oplus P_{7(r+2)+5}\stackrel{(\a\text{ }\b^{2})}
\longrightarrow P_{7(r+2)}\stackrel{\g\b^{2}}\longrightarrow\\
\longrightarrow P_{7(r+1)+4}\stackrel{\b\g}\longrightarrow P_{7r+6}
\stackrel{\a}\longrightarrow P_{7r+3}\longrightarrow S_{7r+3}\longrightarrow 0.
\end{multline*}
At that $\Omega^{13}(S_{7r+3})\simeq S_{7(r+7)+1}$.
\end{lem}

\begin{lem}
The begin of the minimal projective resolution of $S_{7r+4}$ is of the form $$\dots\longrightarrow
P_{7r+5} \stackrel{\b}\longrightarrow P_{7r+4}\longrightarrow S_{7r+4}\longrightarrow 0.$$ At that
$\Omega^{2}(S_{7r+4})\simeq S_{7(r+1)+5}$.
\end{lem}

\begin{lem}
The begin of the minimal projective resolution of $S_{7r+5}$ is of the form
\begin{multline*}
\dots\longrightarrow P_{7(r+7)+4}\stackrel{\b}\longrightarrow
P_{7(r+7)}\stackrel{\g\a}\longrightarrow
P_{7(r+6)+3}\stackrel{\a^{3}\g}\longrightarrow\\
\longrightarrow
P_{7(r+5)+6}\stackrel{\binom{\a^{2}}{-\b}}\longrightarrow
P_{7(r+5)+2}\oplus P_{7(r+5)+5}\stackrel{(\a^{2}\text{ }\b^{2})} \longrightarrow\\
\longrightarrow P_{7(r+5)}\stackrel{\binom{\g\b^{2}}{-\g\a^{3}}}\longrightarrow
P_{7(r+4)+4}\oplus P_{7(r+4)+1}\stackrel{\binom{\phantom{-}\b\g\phantom{-}0}{\b\phantom{-}\a}}\longrightarrow\\
\longrightarrow
P_{7(r+3)+6}\oplus P_{7(r+4)}\stackrel{\binom{\phantom{-}\a\phantom{-}0}{-\b\phantom{-}\g\b}}\longrightarrow
P_{7(r+3)+3}\oplus P_{7(r+3)+5}\stackrel{(\a^{3}\g\text{ }\b^{2}\g)}\longrightarrow\\
\longrightarrow P_{7(r+2)+6}\stackrel{\binom{\a^{2}}{-\b^{2}}}\longrightarrow
P_{7(r+2)+2}\oplus P_{7(r+2)+4}\stackrel{(\a^{2}\text{ }\b)}
\longrightarrow P_{7(r+2)}\stackrel{\g\a^{3}}\longrightarrow\\
\longrightarrow P_{7(r+1)+1}\stackrel{\a\g}\longrightarrow P_{7r+6}
\stackrel{\b}\longrightarrow P_{7r+5}\longrightarrow S_{7r+5}\longrightarrow 0.
\end{multline*}
At that $\Omega^{15}(S_{7r+5})\simeq S_{7(r+8)+4}$.
\end{lem}

\begin{lem}\label{lem_s6}
The begin of the minimal projective resolution of $S_{7r+6}$ is of the form $$\dots\longrightarrow
P_{7r+7} \stackrel{\g}\longrightarrow P_{7r+6}\longrightarrow S_{7r+6}\longrightarrow 0.$$ At that
$\Omega^{2}(S_{7r+6})\simeq S_{7(r+2)}$.
\end{lem}

\begin{proof}
Proofs of the lemmas consist of direct check that given sequences are exact, and it is immediate.
\end{proof}
We shall need the Happel's lemma (see \cite{Ha}), as revised in \cite{Gen&Ka}:

\begin{lem}[Happel]\label{lem_Ha}
Let
$$\dots\rightarrow Q_m\rightarrow Q_{m-1}\rightarrow\dots\rightarrow
Q_1\rightarrow Q_0\rightarrow R\rightarrow 0$$ be the minimal projective resolution of $R$. Then
$$Q_m\cong\bigoplus_{i,j}P_{i,j}^{\dim\mathrm{Ext}^m_R(S_j,S_i)}.$$
\end{lem}

\begin{proof}[Proof of the theorem \ref{resol_thm}]
Descriptions for $Q_i$  immediately follows from lemmas \ref{lem_s0} -- \ref{lem_s6} and Happel's
lemma.

As proved in \cite{VGI}, to prove that sequence \eqref{resolv} is exact in $Q_m$ ($m\le 17$)
it will be sufficient to show that $d_md_{m+1}=0$. It is easy to verify this relation by
a straightforward calculation of matrixes products.

Since the sequence is exact in $Q_{17}$, it follows that
$\Omega^{17}({}_\Lambda R)\simeq {}_1R_{\sigma}$, where
$\Omega^{17}({}_\Lambda R)=\Im d_{16}$ is the 17th syzygy of the
module $R$, and ${}_1R_{\sigma}$ is a twisted bimodule. Hence, an
exactness in $Q_t$ ($t>17$) holds.

\end{proof}

We recall that for $R$-bimodule $M$ the {\it twisted bimodule} is a linear
space $M$, on which left act right acts of the algebra $R$ (denoted by
asterisk) are assigned by the following way:
$$r*m*s = \lambda(r)\cdot m\cdot\mu(s) \text{ for } r,s\in R
\text{ and } m\in M,$$ where $\lambda,\mu$ are some automorphisms of algebra $R$.
Such twisted bimodule we shall denote by ${}_\lambda M_\mu$.

\begin{s}
We have isomorphism $\Omega^{17}({}_\Lambda R)\simeq {}_1R_{\sigma}$.
\end{s}

\begin{pr}
Automorphism $\sigma$ has a finite order, and

$(1)$ if $\myChar=2$, then order of $\sigma$ is equal to $\frac{s}{\myNod(s,9)}$;

$(2)$ if $\myChar\ne 2$, then order of $\sigma$ is equal to $\frac{s}{\myNod(s,9)}$, if
$\frac{s}{\myNod(s,9)}$ is even, and to $\frac{2s}{\myNod(s,9)}$ otherwise.
\end{pr}

\begin{pr}
The minimal period of bimodule resolution of $R$ is $17\ord\sigma$.
\end{pr}

\section{Tree class $E_7$: The additive structure of $\HH^*(R)$}

\begin{pr}[Dimensions of homomorphism groups, $s>1$]\label{dim_hom}
Let $s>1$ and $R=R_s$ is algebra of the type $E_7$. Next, $deg\in\N\cup\{0\}$,
$\ell$ be the aliquot, and $r$ be the residue of division of $deg$ by $17$.

$(1)$ If $r\in\{0,8,16\}$, then $$\dim_K\Hom_\Lambda(Q_{deg}, R)=\begin{cases}
 7s,\quad m+9\ell\equiv 0(s)\text{ or }m+9\ell\equiv 1(s);\\
 0,\quad\text{otherwise.}
 \end{cases}$$

$(2)$ If $r\in\{1,15\}$, then $$\dim_K\Hom_\Lambda(Q_{deg}, R)=\begin{cases}
 8s,\quad m+9\ell\equiv 0(s);\\
 0,\quad\text{otherwise.}
 \end{cases}$$

$(3)$ If $r=2$, then $$\dim_K\Hom_\Lambda(Q_{deg},
R)=\begin{cases}
 4s,\quad m+9\ell\equiv 0(s);\\
 s,\quad m+9\ell\equiv 1(s);\\
 0,\quad\text{otherwise.}
 \end{cases}$$

$(4)$ If $r\in\{3,13\}$, then $$\dim_K\Hom_\Lambda(Q_{deg},
R)=\begin{cases}
 9s,\quad m+9\ell\equiv 0(s);\\
 0,\quad\text{otherwise.}
 \end{cases}$$

$(5)$ If $r=4$, then $$\dim_K\Hom_\Lambda(Q_{deg},
R)=\begin{cases}
 3s,\quad m+9\ell\equiv 0(s);\\
 5s,\quad m+9\ell\equiv 1(s);\\
 0,\quad\text{otherwise.}
 \end{cases}$$

$(6)$ If $r\in\{5,11\}$, then $$\dim_K\Hom_\Lambda(Q_{deg},
R)=\begin{cases}
 10s,\quad m+9\ell\equiv 0(s);\\
 0,\quad\text{otherwise.}
 \end{cases}$$

$(7)$ If $r=6$, then $$\dim_K\Hom_\Lambda(Q_{deg},
R)=\begin{cases}
 5s,\quad m+9\ell\equiv 0(s);\\
 7s,\quad m+9\ell\equiv 1(s);\\
 0,\quad\text{otherwise.}
 \end{cases}$$

$(8)$ If $r\in\{7,9\}$, then $$\dim_K\Hom_\Lambda(Q_{deg},
R)=\begin{cases}
 12s,\quad m+9\ell\equiv 0(s);\\
 0,\quad\text{otherwise.}
 \end{cases}$$
 
$(9)$ If $r=10$, then $$\dim_K\Hom_\Lambda(Q_{deg},
R)=\begin{cases}
 7s,\quad m+9\ell\equiv 0(s);\\
 5s,\quad m+9\ell\equiv 1(s);\\
 0,\quad\text{otherwise.}
 \end{cases}$$
 
$(10)$ If $r=12$, then $$\dim_K\Hom_\Lambda(Q_{deg},
R)=\begin{cases}
 5s,\quad m+9\ell\equiv 0(s);\\
 3s,\quad m+9\ell\equiv 1(s);\\
 0,\quad\text{otherwise.}
 \end{cases}$$
 
$(11)$ If $r=14$, then $$\dim_K\Hom_\Lambda(Q_{deg},
R)=\begin{cases}
 s,\quad m+9\ell\equiv 0(s);\\
 4s,\quad m+9\ell\equiv 1(s);\\
 0,\quad\text{otherwise.}
 \end{cases}$$
 
\end{pr}

\begin{proof}
The dimension $\dim_K\Hom_\Lambda(P_{i,j}, R)$ is equal to the number of linear independent nonzero
paths of the quiver $\mathcal{Q}_s$, leading from $j$th vertex to $i$th, and the proof is to
consider cases $r=0$, $r=1$ etc.
\end{proof}

\begin{pr}[Dimensions of homomorphism groups, $s=1$]

Let $R=R_1$ is algebra of the type $E_7$. Next, $deg\in\N\cup\{0\}$,
$\ell$ be the aliquot, and $r$ be the residue of division of $deg$ by $17$.

$(1)$ If $r\in\{0,8,16\}$, then $\dim_K\Hom_\Lambda(Q_{deg}, R)=14$.

$(2)$ If $r\in\{1,4,12,15\}$, then $\dim_K\Hom_\Lambda(Q_{deg}, R)=8$.

$(3)$ If $r\in\{2,14\}$, then $\dim_K\Hom_\Lambda(Q_{deg}, R)=5$.

$(4)$ If $r\in\{3,13\}$, then $\dim_K\Hom_\Lambda(Q_{deg}, R)=9$.

$(5)$ If $r\in\{5,11\}$, then $\dim_K\Hom_\Lambda(Q_{deg}, R)=10$.

$(6)$ If $r\in\{6,7,9,10\}$, then $\dim_K\Hom_\Lambda(Q_{deg}, R)=12$.

\end{pr}

\begin{proof}
The proof is basically the same as proof of proposition \ref{dim_hom}.
\end{proof}

\begin{pr}[Dimensions of coboundaries groups]\label{dim_im}
Let $R=R_s$ is algebra of the type $E_7$, and let
\begin{equation}\tag{$\times$}\label{ind_resolv} 0\longrightarrow
\Hom_\Lambda(Q_0, R)\stackrel{\delta^0}\longrightarrow \Hom_\Lambda(Q_1,
R)\stackrel{\delta^1}\longrightarrow \Hom_\Lambda(Q_2, R)\stackrel{\delta^2}\longrightarrow\dots
\end{equation} be a complex, obtained from minimal projective
resolution \eqref{resolv} of algebra $R$, by applying functor
$\Hom_\Lambda(-,R)$.

Consider coboundaries groups $\Im\delta^{deg}$ of the complex
\eqref{ind_resolv}. Let $\ell$ be the aliquot, and $r$ be the
residue of division of $deg$ by $17$, $m$ be the aliquot of
division of $r$ by $2$. Then$:$

$(1)$ If $r\in\{0,7,8,15,16\}$, then $$\dim_K\Im\delta^{deg}=\begin{cases}
 7s-1,\quad m+9\ell\equiv 0(s),\text{ }\ell+m\div 2\text{ or }\myChar=2;\\
 7s,\quad m+9\ell\equiv 0(s),\text{ }\ell+m\ndiv 2,\text{ }\myChar\ne 2;\\
 0,\quad\text{otherwise.}
 \end{cases}$$

$(2)$ If $r\in\{1,14\}$, then $$\dim_K\Im\delta^{deg}=\begin{cases}
 s,\quad m+9\ell\equiv 0(s);\\
 0,\quad\text{otherwise.}
 \end{cases}$$

$(3)$ If $r\in\{2,13\}$, then $$\dim_K\Im\delta^{deg}=\begin{cases}
 4s,\quad m+9\ell\equiv 0(s);\\
 0,\quad\text{otherwise.}
 \end{cases}$$

$(4)$ If $r\in\{3,12\}$, then $$\dim_K\Im\delta^{deg}=\begin{cases}
 5s-1,\quad m+9\ell\equiv 0(s),\text{ }\ell+m\div 2\text{ or }\myChar=2;\\
 5s,\quad m+9\ell\equiv 0(s),\text{ }\ell+m\ndiv 2,\text{ }\myChar\ne 2;\\
 0,\quad\text{otherwise.}
 \end{cases}$$

$(5)$ If $r\in\{4,11\}$, then $$\dim_K\Im\delta^{deg}=\begin{cases}
 3s,\quad m+9\ell\equiv 0(s);\\
 0,\quad\text{otherwise.}
 \end{cases}$$

$(6)$ If $r\in\{5,10\}$, then $$\dim_K\Im\delta^{deg}=\begin{cases}
 7s-1,\quad m+9\ell\equiv 0(s),\text{ }\ell+m\div 2\text{ and }\myChar=3;\\
 7s,\quad m+9\ell\equiv 0(s),\text{ }\ell+m\ndiv 2\text{ or }\myChar\ne 3;\\
 0,\quad\text{otherwise.}
 \end{cases}$$

$(7)$ If $r\in\{6,9\}$, then $$\dim_K\Im\delta^{deg}=\begin{cases}
 5s,\quad m+9\ell\equiv 0(s);\\
 0,\quad\text{otherwise.}
 \end{cases}$$
 
\end{pr}

\begin{proof}
The proof is technical and consists in constructing the image matrixes from the description of
differential matrixes and the subsequent computations of the ranks of image matrixes.
\end{proof}

\begin{thm}[Additive structure, $s>1$]
Let $s>1$ and $R=R_s$ is algebra of the type $E_7$. Next, $deg\in\N\cup\{0\}$,
$\ell$ be the aliquot, and $r$ be the residue of division of $deg$ by $17$,
$m$ be the aliquot of division of $r$ by $2$. Then
$\dim_K\HH^{deg}(R)=1$, if one of the following conditions takes place$:$

$(1)$ $r\in \{0,1,3,7,8,9,12,13,15,16\}$, $m+9\ell\equiv 0(s),\text{ }\ell+m\div 2\text{ or }\myChar=2$;

$(2)$ $r\in \{0,4,8,16\}$, $m+9\ell\equiv 1(s),\text{ }\ell+m\ndiv 2\text{ or }\myChar=2$;

$(3)$ $r=6$, $m+9\ell\equiv 1(s),\text{ }\ell+m\ndiv 2,\text{ }\myChar=3$;

$(4)$ $r\in \{5,10,11\}$, $m+9\ell\equiv 0(s),\text{ }\ell+m\div 2,\text{ }\myChar=3$.

In other cases $\dim_K\HH^{deg}(R)=0$.
\end{thm}

\begin{proof}
As $\dim_K\HH^{deg}(R)=\dim_K\Ker\delta^{deg}-\dim_K\Im\delta^{deg-1}$, and
$\dim_K\Ker\delta^{deg}=\dim_K\Hom_\Lambda(Q_{deg},R)-\dim_K\Im\delta^{deg}$,
the assertions of theorem easily follows from propositions
\ref{dim_hom} -- \ref{dim_im}.
\end{proof}

\begin{thm}[Additive structure, $s=1$]
Let $R=R_1$ is algebra of the type $E_7$. Next, $deg\in\N\cup\{0\}$,
$\ell$ be the aliquot, and $r$ be the residue of division of $deg$ by $17$.\\
$($a$)$ $\dim_K\HH^{deg}(R)=8$, if $deg=0$.\\
$($b$)$ $\dim_K\HH^{deg}(R)=2$, if $deg>0$, $r\in \{0,8,16\}$, $\myChar=2$.\\
$($c$)$ $\dim_K\HH^{deg}(R)=1$, if one of the following conditions takes
place$:$

$(1)$ $r\in \{0,8,16\}$, $deg>0$, $\myChar\ne 2$;

$(2)$ $r\in \{1,3,7,9,12,13,15\}$, $\ell+m\div 2$ or $\myChar=2$;

$(3)$ $r=4$, $\ell+m\ndiv 2$ or $\myChar=2$;

$(4)$ $r=6$, $\ell+m\ndiv 2$, $\myChar=3$;

$(5)$ $r\in \{5,10,11\}$, $\ell+m\div 2$, $\myChar=3$.\\ $($d$)$ In other cases
$\dim_K\HH^{deg}(R)=0$.

\end{thm}

\section{Tree class $E_7$: Generators of $\HH^*(R)$}\label{gen}

For $s>1$ introduce the set of generators $Y^{(1)}_t$, $Y^{(2)}_t$, \dots $Y^{(18)}_t$, such that
$\deg Y_t^{(i)}=t$, $0\le t < 17\ord\sigma$ and $t$ satisfies conditions of (i)th item from the
list on page \pageref{degs}. For $s=1$ introduce the set of
generators $Y^{(1)}_t$, $Y^{(2)}_t$, \dots
$Y^{(25)}_t$, such that $\deg Y_t^{(i)}=t$, $0\le t < 17\ord\sigma$ and $t$ satisfies
conditions of (i)th item from the list on page \pageref{degs} for
$i\le 18$ and $t=0$ if $i>18$. For the generator element $Q_t\rightarrow R$ 
we shall describe the map $Q_t\rightarrow Q_0$ as matrix. The corresponding generator element is an composition
of the map with multiplication map $Q_t\rightarrow Q_0\stackrel\varepsilon\longrightarrow R$.

(1) $Y^{(1)}_t$ is a $(7s\times 7s)$-matrix, whose elements $y_{ij}$ have the following form:
$$y_{ij}=
\begin{cases}
e_{7j+j_2}\otimes e_{7j+j_2},\quad i=j;\\
0,\quad\text{otherwise.}\end{cases}$$

(2) $Y^{(2)}_t$ is a $(7s\times 7s)$-matrix with a single nonzero
element:
$$y_{0,0}=w_{0\ra 7}\otimes e_{0}.$$

(3) $Y^{(3)}_t$ is a $(7s\times 8s)$-matrix with two nonzero
elements:
$$y_{0,0}=w_{0\ra 1}\otimes e_{0}\text{ and } y_{0,s}=w_{0\ra 4}\otimes e_{0}.$$

(4) $Y^{(4)}_t$ is a $(7s\times 9s)$-matrix, whose elements $y_{ij}$ have the following form:

If $0\le j<s$, then $$y_{ij}=
\begin{cases}w_{7j\ra 7j+2}\otimes e_{7j},\quad i=j;\\
0,\quad\text{otherwise.}\end{cases}$$

If $s\le j<2s$, then $$y_{ij}=
\begin{cases}-w_{7j\ra 7j+5}\otimes e_{7j},\quad i=j-s;\\
0,\quad\text{otherwise.}\end{cases}$$

If $2s\le j<4s$, then $y_{ij}=0.$

If $4s\le j<5s$, then $$y_{ij}=
\begin{cases}w_{7j+3\ra 7(j+1)}\otimes e_{7j+3},\quad i=j-s;\\
0,\quad\text{otherwise.}\end{cases}$$

If $5s\le j<7s$, then $y_{ij}=0.$

If $7s\le j<8s$, then $$y_{ij}=
\begin{cases}w_{7j+6\ra 7(j+1)+1}\otimes e_{7j+6},\quad i=j-s;\\
0,\quad\text{otherwise.}\end{cases}$$

If $8s\le j<9s$, then $y_{ij}=0.$

(5) $Y^{(5)}_t$ is a $(7s\times 10s)$-matrix with a single nonzero element:
$$y_{0,0}=w_{0\ra 6}\otimes e_{0}.$$

(6) $Y^{(6)}_t$ is a $(7s\times 10s)$-matrix, whose elements $y_{ij}$ have the following form:

If $0\le j<s$, then $$y_{ij}=
\begin{cases}w_{7j\ra 7j+4}\otimes e_{7j},\quad i=j;\\
0,\quad\text{otherwise.}\end{cases}$$

If $s\le j<2s$, then $$y_{ij}=
\begin{cases}-w_{7j\ra 7j+3}\otimes e_{7j},\quad i=j-s;\\
0,\quad\text{otherwise.}\end{cases}$$

If $2s\le j<6s$, then $y_{ij}=0.$

If $6s\le j<7s$, then $$y_{ij}=
\begin{cases}w_{7j+4\ra 7(j+1)}\otimes e_{7j+4},\quad i=j-2s;\\
0,\quad\text{otherwise.}\end{cases}$$

If $7s\le j<8s$, then $$y_{ij}=
\begin{cases}w_{7j+5\ra 7j+6}\otimes e_{7j+5},\quad i=j-2s;\\
0,\quad\text{otherwise.}\end{cases}$$

If $8s\le j<9s$, then $y_{ij}=0.$

If $9s\le j<10s$, then $$y_{ij}=
\begin{cases}-w_{7j+6\ra 7(j+1)+5}\otimes e_{7j+6},\quad i=j-3s;\\
0,\quad\text{otherwise.}\end{cases}$$

(7) $Y^{(7)}_t$ is a $(7s\times 12s)$-matrix with a single nonzero element:
$$y_{6s,10s}=w_{7j+6\ra 7(j+1)+6}\otimes e_{7j+6}.$$

(8) $Y^{(8)}_t$ is a $(7s\times 12s)$-matrix, whose elements $y_{ij}$ have the following form:

If $0\le j<2s$, then $y_{ij}=0.$

If $2s\le j<3s$, then $$y_{ij}=
\begin{cases}w_{7j\ra 7j+4}\otimes e_{7j},\quad i=j-2s;\\
0,\quad\text{otherwise.}\end{cases}$$

If $3s\le j<7s$, then $y_{ij}=0.$

If $7s\le j<8s$, then $$y_{ij}=
\begin{cases}w_{7j+5\ra 7j+6}\otimes e_{7j+5},\quad i=j-2s;\\
0,\quad\text{otherwise.}\end{cases}$$

If $8s\le j<9s$, then $y_{ij}=0.$

If $9s\le j<10s$, then $$y_{ij}=
\begin{cases}w_{7j+6\ra 7(j+1)+4}\otimes e_{7j+6},\quad i=j-3s;\\
0,\quad\text{otherwise.}\end{cases}$$

If $10s\le j<12s$, then $y_{ij}=0.$

(9) $Y^{(9)}_t$ is a $(7s\times 13s)$-matrix, whose elements $y_{ij}$ have the following form:

If $0\le j<s$, then $y_{ij}=0.$

If $s\le j<2s$, then $$y_{ij}=
\begin{cases}e_{7j}\otimes e_{7j},\quad i=j-s;\\
0,\quad\text{otherwise.}\end{cases}$$

If $2s\le j<3s$, then $$y_{ij}=
\begin{cases}-e_{7j+1}\otimes e_{7j+1},\quad i=j-s;\\
0,\quad\text{otherwise.}\end{cases}$$

If $3s\le j<4s$, then $y_{ij}=0.$

If $4s\le j<5s$, then $$y_{ij}=
\begin{cases}-e_{7j+2}\otimes e_{7j+2},\quad i=j-2s;\\
0,\quad\text{otherwise.}\end{cases}$$

If $5s\le j<6s$, then $$y_{ij}=
\begin{cases}-e_{7j+3}\otimes e_{7j+3},\quad i=j-2s;\\
0,\quad\text{otherwise.}\end{cases}$$

If $6s\le j<8s$, then $y_{ij}=0.$

If $8s\le j<9s$, then $$y_{ij}=
\begin{cases}w_{7j+4\ra 7j+5}\otimes e_{7j+4},\quad i=j-4s;\\
0,\quad\text{otherwise.}\end{cases}$$

If $9s\le j<11s$, then $y_{ij}=0.$

If $11s\le j<12s$, then $$y_{ij}=
\begin{cases}e_{7j+6}\otimes e_{7j+6},\quad i=j-5s;\\
0,\quad\text{otherwise.}\end{cases}$$

If $12s\le j<13s$, then $$y_{ij}=
\begin{cases}-w_{7j+6\ra 7(j+1)}\otimes e_{7j+6},\quad i=j-6s;\\
0,\quad\text{otherwise.}\end{cases}$$

(10) $Y^{(10)}_t$ is a $(7s\times 13s)$-matrix with a single nonzero element:
$$y_{0,0}=-w_{0\ra 6}\otimes e_{0}.$$

(11) $Y^{(11)}_t$ is a $(7s\times 12s)$-matrix with three nonzero elements:
$$y_{0,0}=w_{0\ra 5}\otimes e_{0},\quad y_{0,2s}=w_{0\ra 1}\otimes e_{0},\quad y_{j-4s,10s+(-1)_s}=w_{7j+6\ra 7(j+1)+5}\otimes e_{7j+6}.$$

(12) $Y^{(12)}_t$ is a $(7s\times 12s)$-matrix, whose elements $y_{ij}$ have the following form:

If $0\le j<s$, then $y_{ij}=0.$

If $s\le j<2s$, then $$y_{ij}=
\begin{cases}e_{7j}\otimes e_{7j},\quad i=j-s;\\
0,\quad\text{otherwise.}\end{cases}$$

If $2s\le j<3s$, then $$y_{ij}=
\begin{cases}w_{7j+1\ra 7j+2}\otimes e_{7j+1},\quad i=j-s;\\
0,\quad\text{otherwise.}\end{cases}$$

If $3s\le j<4s$, then $$y_{ij}=
\begin{cases}w_{7j+2\ra 7j+3}\otimes e_{7j+2},\quad i=j-s;\\
0,\quad\text{otherwise.}\end{cases}$$

If $4s\le j<6s$, then $y_{ij}=0.$

If $6s\le j<7s$, then $$y_{ij}=
\begin{cases}e_{7j+4}\otimes e_{7j+4},\quad i=j-2s;\\
0,\quad\text{otherwise.}\end{cases}$$

If $7s\le j<9s$, then $y_{ij}=0.$

If $9s\le j<10s$, then $$y_{ij}=
\begin{cases}e_{7j+5}\otimes e_{7j+5},\quad i=j-4s;\\
0,\quad\text{otherwise.}\end{cases}$$

If $10s\le j<11s$, then $$y_{ij}=
\begin{cases}e_{7j+6}\otimes e_{7j+6},\quad i=j-4s;\\
0,\quad\text{otherwise.}\end{cases}$$

If $11s\le j<12s$, then $y_{ij}=0.$

(13) $Y^{(13)}_t$ is a $(7s\times 10s)$-matrix with two nonzero elements:
$$y_{s,2s}=w_{1\ra 7}\otimes e_{1}\text{ and }
 y_{4s,5s}=w_{7j+4\ra 7j+7}\otimes e_{7j+4}.$$

(14) $Y^{(14)}_t$ is a $(7s\times 10s)$-matrix, whose elements $y_{ij}$ have the following form:

If $0\le j<s$, then $$y_{ij}=
\begin{cases}e_{7j}\otimes e_{7j},\quad i=j;\\
0,\quad\text{otherwise.}\end{cases}$$

If $s\le j<2s$, then $$y_{ij}=
\begin{cases}w_{7j+1\ra 7j+3}\otimes e_{7j+1},\quad i=j;\\
0,\quad\text{otherwise.}\end{cases}$$

If $2s\le j<6s$, then $y_{ij}=0.$

If $6s\le j<7s$, then $$y_{ij}=
\begin{cases}-w_{7j+4\ra 7j+5}\otimes e_{7j+4},\quad i=j-2s;\\
0,\quad\text{otherwise.}\end{cases}$$

If $7s\le j<8s$, then $y_{ij}=0.$

If $8s\le j<9s$, then $$y_{ij}=
\begin{cases}e_{7j+6}\otimes e_{7j+6},\quad i=j-2s;\\
0,\quad\text{otherwise.}\end{cases}$$

If $9s\le j<10s$, then $y_{ij}=0.$

(15) $Y^{(15)}_t$ is a $(7s\times 9s)$-matrix with two nonzero elements:
$$y_{0,0}=w_{0\ra 3}\otimes e_{0}\text{ and }y_{0,s}=w_{0\ra 5}\otimes e_{0}.$$

(16) $Y^{(16)}_t$ is a $(7s\times 8s)$-matrix, whose elements $y_{ij}$ have the following form:

If $0\le j<s$, then $$y_{ij}=
\begin{cases}w_{7j\ra 7j+6}\otimes e_{7j},\quad i=j;\\
0,\quad\text{otherwise.}\end{cases}$$

If $s\le j<4s$, then $y_{ij}=0.$

If $4s\le j<5s$, then $$y_{ij}=
\begin{cases}-w_{7j+4\ra 7(j+1)}\otimes e_{7j+4},\quad i=j;\\
0,\quad\text{otherwise.}\end{cases}$$

If $5s\le j<6s$, then $$y_{ij}=
\begin{cases}w_{7j+5\ra 7(j+1)+4}\otimes e_{7j+5},\quad i=j;\\
0,\quad\text{otherwise.}\end{cases}$$

If $6s\le j<7s$, then $y_{ij}=0.$

If $7s\le j<8s$, then $$y_{ij}=
\begin{cases}-w_{7j+6\ra 7(j+1)+5}\otimes e_{7j+6},\quad i=j-s;\\
0,\quad\text{otherwise.}\end{cases}$$

(17) $Y^{(17)}_t$ is a $(7s\times 7s)$-matrix, whose elements $y_{ij}$ have the following form:

If $0\le j<4s$, then $$y_{ij}=
\begin{cases}e_{7j+j_2}\otimes e_{7j+j_2},\quad i=j;\\
0,\quad\text{otherwise.}\end{cases}$$

If $4s\le j<6s$, then $y_{ij}=0.$

If $6s\le j<7s$, then $$y_{ij}=
\begin{cases}e_{7j+6}\otimes e_{7j+6},\quad i=j;\\
0,\quad\text{otherwise.}\end{cases}$$

(18) $Y^{(18)}_t$ is a $(7s\times 7s)$-matrix with a single nonzero element:
$$y_{0,0}=-w_{0\ra 7}\otimes e_{0}.$$

(19) $Y^{(19)}_t$ is a $(7s\times 7s)$-matrix with a single nonzero element:
$$y_{0,0}=w_{0\ra 7}\otimes e_{0}.$$

(20) $Y^{(20)}_t$ is a $(7s\times 7s)$-matrix with a single nonzero element:
$$y_{4s,4s}=w_{4\ra 11}\otimes e_{4}.$$

(21) $Y^{(21)}_t$ is a $(7s\times 7s)$-matrix with a single nonzero element:
$$y_{5s,5s}=w_{5\ra 12}\otimes e_{5}.$$

(22) $Y^{(22)}_t$ is a $(7s\times 7s)$-matrix with a single nonzero element:
$$y_{3s,3s}=w_{3\ra 10}\otimes e_{3}.$$

(23) $Y^{(23)}_t$ is a $(7s\times 7s)$-matrix with a single nonzero element:
$$y_{s,s}=w_{1\ra 8}\otimes e_{1}.$$

(24) $Y^{(24)}_t$ is a $(7s\times 7s)$-matrix with a single nonzero element:
$$y_{2s,2s}=w_{2\ra 9}\otimes e_{2}.$$

(25) $Y^{(25)}_t$ is a $(7s\times 7s)$-matrix with a single nonzero element:
$$y_{6s,6s}=w_{6\ra 13}\otimes e_{6}.$$

\section{Tree class $E_7$: Multiplications in $\HH^*(R)$}

Let $Q_\bullet\rightarrow R$ be the minimal projective bimodule
resolution of the algebra $R$, constructed in paragraph
\ref{sect_res}. Any $t$-cocycle $f\in\Ker\delta^t$ is lifted
(uniquely up to homotopy) to a chain map of complexes $\{\varphi_i:
Q_{t+i}\rightarrow Q_i\}_{i\ge 0}$. The homomorphism $\varphi_i$ is
called the {\it $i$th translate} of the cocycle $f$ and will be
denoted by $\Omega^i(f)$. For cocycles $f_1\in\Ker\delta^{t_1}$ and
$f_2\in\Ker\delta^{t_2}$ we have
\begin{equation}\tag{$*$}\label{mult_formula}
\cl f_2\cdot \cl f_1=\cl(\Omega^0(f_2)\Omega^{t_2}(f_1)).
\end{equation}

From the descriptions of elements $Y^{(i)}_t$ (given in Sec. \ref{gen}) and
its $\Omega$-translates (see in ancillary files of this paper) we can find multiplications of the elements
using the formula \eqref{mult_formula}.

We will find a multiplication of elements of the types 14 and 3 for
$s>1$.

Consider two arbitrary elements $Y_{t_{14}}^{(14)}$ and $Y_{t_3}^{(3)}$.
For its degrees $t_{14}$ and $t_3$ we have:
\begin{align*}
t_{14}&=17\ell_{14}+12,\text{ }6+9\ell_{14}\equiv 0(s),\text{ }\ell_{14}\div 2\text{ or }\myChar=2;\\
t_3&=17\ell_3+1,\text{ }9\ell_3\equiv 0(s),\text{ }\ell_3\div 2\text{ or }\myChar=2.
\end{align*}
Let $t=t_{14}+t_3$; this is the degree of an element
$Y_{t_{14}}^{(14)}Y_{t_3}^{(3)}$. Then $t=17(\ell_{14}+\ell_3)+13$. Group of
the degree $t$ has type (15). $Y^{(3)}_t$ is a $(7s\times 8s)$-matrix with two nonzero elements
$y_{0,0}=w_{0\ra 1}\otimes e_0$ and $y_{0,s}=w_{0\ra 4}\otimes e_0$. 
$\Omega^{t_3}(Y_{t_{14}}^{(14)})$ is an $(8s\times 9s)$-matrix that was described in proposition 11 
(in the corresponding ancillary file).
Multiplication of $\Omega^{t_3}(Y_{t_{14}}^{(14)})$ and $Y^{(3)}_t$ an $(7s\times 9s)$-matrix 
with the following nonzero elements:
$$b_{0,0}=w_{0\ra 3}\otimes e_0,\quad b_{0,s}=w_{0\ra 5}\otimes e_0.$$
This matrix is the same as $Y^{(15)}_t$.


Multiplications of other elements, except $Y^{(5)}$, $Y^{(10)}$, $Y^{(11)}$, $Y^{(15)}$ and
$Y^{(18)}$, are similarly considered. To get the whole picture we should prove the following lemma.

\begin{lem}$\text{ }$

$($a$)$ Let $Y^{(5)}$ be an arbitrary element from generators of the
corresponding type. Then there are elements $Y^{(3)}$ and $Y^{(4)}$
such as $Y^{(5)}=Y^{(3)}Y^{(4)}$.

$($b$)$ Let $Y^{(10)}$ be an arbitrary element from generators of
the corresponding type. Then there are elements $Y^{(3)}$ and
$Y^{(8)}$ such as $Y^{(10)}=Y^{(3)}Y^{(8)}$.

$($c$)$ Let $Y^{(11)}$ be an arbitrary element from generators of
the corresponding type. Then there are elements $Y^{(3)}$ and
$Y^{(9)}$ such as $Y^{(11)}=Y^{(3)}Y^{(9)}$.

$($d$)$ Let $Y^{(15)}$ be an arbitrary element from generators of
the corresponding type. Then there are elements $Y^{(3)}$ and
$Y^{(14)}$ such as $Y^{(15)}=Y^{(3)}Y^{(14)}$.

$($e$)$ Let $Y^{(18)}$ be an arbitrary element from generators of
the corresponding type. Then there are elements $Y^{(3)}$ and
$Y^{(16)}$ such as $Y^{(18)}=Y^{(3)}Y^{(16)}$.

\end{lem}

\begin{proof}
The degree 1 has type 3, for all $s$. It only remains to use the
relations for type (3).
\end{proof}

\section{Tree class $E_8$: Statement of the main results}\label{res8}

Let $\HH^t(R)$ be the $t$th group of the Hochschild cohomology ring of $R$ with coefficients in
$R$. Let $\ell$ and $r$ be the integral part and residue of $t$ modulo 29, 
and $m$ the integral part of $r$ modulo $2$.

Consider the case of $s>1$. To describe Hochschild cohomology ring of the algebra $R^\prime_s$ we
introduce the following conditions for arbitrary degree $t$:

$($1$)$ $r=0$, $m+15\ell\equiv 0(s),\text{ }\ell\div 2\text{ or }\myChar=2$;

$($2$)$ $r=1$, $m+15\ell\equiv 0(s),\text{ }\ell\div 2\text{ or }\myChar=2$;

$($3$)$ $r=3$, $m+15\ell\equiv 0(s),\text{ }\ell\ndiv 2\text{ or }\myChar=2$;

$($4$)$ $r=4$, $m+15\ell\equiv 1(s),\text{ }\ell\ndiv 2\text{ or }\myChar=2$;

$($5$)$ $r=5$, $m+15\ell\equiv 0(s),\text{ }\ell\div 2,\text{ }\myChar=3$;

$($6$)$ $r=6$, $m+15\ell\equiv 1(s),\text{ }\ell\div 2,\text{ }\myChar=3$;

$($7$)$ $r=7$, $m+15\ell\equiv 0(s),\text{ }\ell\ndiv 2\text{ or }\myChar=2$;

$($8$)$ $r=8$, $m+15\ell\equiv 1(s),\text{ }\ell\ndiv 2\text{ or }\myChar=2$;

$($9$)$ $r=9$, $m+15\ell\equiv 0(s),\text{ }\ell\div 2,\text{ }\myChar=5$;

$($10$)$ $r=10$, $m+15\ell\equiv 0(s),\text{ }\ell\ndiv 2,\text{ }\myChar=3$;

$($11$)$ $r=10$, $m+15\ell\equiv 1(s),\text{ }\ell\div 2,\text{ }\myChar=5$;

$($12$)$ $r=11$, $m+15\ell\equiv 0(s),\text{ }\ell\ndiv 2,\text{ }\myChar=3$;

$($13$)$ $r=12$, $m+15\ell\equiv 0(s),\text{ }\ell\div 2\text{ or }\myChar=2$;

$($14$)$ $r=13$, $m+15\ell\equiv 0(s),\text{ }\ell\div 2\text{ or }\myChar=2$;

$($15$)$ $r=15$, $m+15\ell\equiv 0(s),\text{ }\ell\ndiv 2\text{ or }\myChar=2$;

$($16$)$ $r=16$, $m+15\ell\equiv 1(s),\text{ }\ell\ndiv 2\text{ or }\myChar=2$;

$($17$)$ $r=17$, $m+15\ell\equiv 0(s),\text{ }\ell\div 2,\text{ }\myChar=3$;

$($18$)$ $r=18$, $m+15\ell\equiv 0(s),\text{ }\ell\ndiv 2,\text{ }\myChar=5$;

$($19$)$ $r=18$, $m+15\ell\equiv 1(s),\text{ }\ell\div 2,\text{ }\myChar=3$;

$($20$)$ $r=19$, $m+15\ell\equiv 0(s),\text{ }\ell\ndiv 2,\text{ }\myChar=5$;

$($21$)$ $r=20$, $m+15\ell\equiv 0(s),\text{ }\ell\div 2\text{ or }\myChar=2$;

$($22$)$ $r=21$, $m+15\ell\equiv 0(s),\text{ }\ell\div 2\text{ or }\myChar=2$;

$($23$)$ $r=22$, $m+15\ell\equiv 0(s),\text{ }\ell\ndiv 2,\text{ }\myChar=3$;

$($24$)$ $r=23$, $m+15\ell\equiv 0(s),\text{ }\ell\ndiv 2,\text{ }\myChar=3$;

$($25$)$ $r=24$, $m+15\ell\equiv 0(s),\text{ }\ell\div 2\text{ or }\myChar=2$;

$($26$)$ $r=25$, $m+15\ell\equiv 0(s),\text{ }\ell\div 2\text{ or }\myChar=2$;

$($27$)$ $r=27$, $m+15\ell\equiv 0(s),\text{ }\ell\ndiv 2\text{ or }\myChar=2$;

$($28$)$ $r=28$, $m+15\ell\equiv 1(s),\text{ }\ell\ndiv 2\text{ or }\myChar=2$;

Let $$M_0=\frac{s}{\myNod(s,15)},\quad M=\begin{cases}29M_0,\quad\myChar=2\text{ or }M_0\div 2;\\58M_0\quad\text{otherwise.}\end{cases}$$

\begin{zam}
In paragraph \ref{sect_res8} we prove that the minimal period of bimodule resolution of $R^\prime_s$
equals $M$.
\end{zam}

Let $\{t_{1, i},\dots,t_{\alpha_i, i} \}$ be a set of all degrees $t$ satisfying the $i$th condition
from the above list, and such that $0\le t_{j, i}<M$ $(j=1,\dots,\alpha_i)$. Let $$\mathcal
X=\bigcup_{i=1}^{28}\left\{X^{(i)}_{t_{j,i}}\right\}_{j=1}^{\alpha_i}\cup\{T\}.$$
We define a grading on the ring $K[\mathcal X]$ of polynomials such that
\begin{align*}\label{degs28}
&\deg X^{(i)}_{t_{j,i}}=t_{j, i} \:\text{for all} \: i=1,\dots,28 \:\text{and}\: j=1,\dots,\alpha_i;\tag{$\circ\circ$}\\
&\deg T=M.
\end{align*}

\begin{zam}\label{brief_notation8}
In the sequel, we often use a simplified notation $X^{(i)}$ for $X^{(i)}_{t_{j,i}}$,
because the values of low indices are clear from the context.
\end{zam}

\begin{obozn}
$$\widetilde X^{(i)}=\begin{cases}X^{(i)},\quad\deg\widetilde
X^{(i)}<\deg T;\\TX^{(i)},\quad\text{otherwise.}\end{cases}$$
\end{obozn}

We define a graded $K$-algebra $\mathcal A=K[\mathcal X]/I$, where $I$ is the ideal generated by
homogeneous elements corresponding to the following relations:
\begin{align*}
&X^{(2)}X^{(2)}=X^{(2)}X^{(4)}=X^{(2)}X^{(5)}=X^{(2)}X^{(6)}=X^{(2)}X^{(8)}=0;\\
&X^{(2)}X^{(9)}=X^{(2)}X^{(10)}=X^{(2)}X^{(11)}=X^{(2)}X^{(12)}=X^{(2)}X^{(14)}=0;\\
&X^{(2)}X^{(16)}=X^{(2)}X^{(17)}=X^{(2)}X^{(18)}=X^{(2)}X^{(19)}=X^{(2)}X^{(20)}=0;\\
&X^{(2)}X^{(22)}=X^{(2)}X^{(23)}=X^{(2)}X^{(24)}=X^{(2)}X^{(26)}=X^{(2)}X^{(28)}=0;\\
&X^{(2)}X^{(1)}=\widetilde X^{(2)},\quad X^{(2)}X^{(3)}=\widetilde X^{(4)},\quad 
X^{(2)}X^{(7)}=\widetilde X^{(8)},\quad X^{(2)}X^{(13)}=\widetilde X^{(14)};\\
&X^{(3)}X^{(15)}=\widetilde X^{(16)},\quad X^{(3)}X^{(21)}=\widetilde X^{(22)},\quad 
X^{(3)}X^{(25)}=\widetilde X^{(26)},\quad X^{(3)}X^{(27)}=\widetilde X^{(28)};
\end{align*}

\begin{align*}
X^{(3)}X^{(3)}&=\begin{cases}s\widetilde X^{(6)},\quad\myChar
=3,\\0,\quad\text{otherwise};\end{cases}&\text{(r1)}\\
X^{(3)}X^{(7)}&=\begin{cases}-2s\widetilde X^{(11)},\quad\myChar
=5,\\0,\quad\text{otherwise};\end{cases}&\text{(r2)}\\
X^{(7)}X^{(13)}&=\begin{cases}s\widetilde X^{(20)},\quad\myChar
=5,\\0,\quad\text{otherwise};\end{cases}&\text{(r3)}\\
X^{(3)}X^{(15)}&=\begin{cases}s\widetilde X^{(19)},\quad\myChar
=3,\\0,\quad\text{otherwise};\end{cases}&\text{(r4)}\\
X^{(3)}X^{(21)}&=\begin{cases}s\widetilde X^{(24)},\quad\myChar
=3,\\0,\quad\text{otherwise};\end{cases}&\text{(r5)}\\
X^{(13)}X^{(27)}&=\begin{cases}2s\widetilde X^{(11)},\quad\myChar
=5,\\0,\quad\text{otherwise};\end{cases}&\text{(r6)}\\
X^{(15)}X^{(21)}&=\begin{cases}s\widetilde X^{(6)},\quad\myChar
=3,\\0,\quad\text{otherwise};\end{cases}&\text{(r7)}\\
X^{(15)}X^{(25)}&=\begin{cases}2s\widetilde X^{(11)},\quad\myChar
=5,\\0,\quad\text{otherwise};\end{cases}&\text{(r8)}\\
X^{(21)}X^{(21)}&=\begin{cases}s\widetilde X^{(12)},\quad\myChar
=3,\\0,\quad\text{otherwise};\end{cases}&\text{(r9)}\\
X^{(21)}X^{(27)}&=\begin{cases}-s\widetilde X^{(19)},\quad\myChar
=3,\\0,\quad\text{otherwise};\end{cases}&\text{(r10)}\\
X^{(25)}X^{(25)}&=\begin{cases}s\widetilde X^{(20)},\quad\myChar
=5,\\0,\quad\text{otherwise}.\end{cases}&\text{(r11)}
\end{align*}

The other relations are described in the tables below (the numbers (r1)--(r11)
in the cells identify the number of relation which describes the product of the corresponding elements):

\setlength{\extrarowheight}{1mm}
\begin{tabular}{c|c|c|c|c|c|c|c|c|c|c}
&$X^{(1)}$&$X^{(3)}$&$X^{(5)}$&$X^{(6)}$&$X^{(7)}$&$X^{(9)}$&$X^{(10)}$&$X^{(11)}$&$X^{(12)}$&$X^{(13)}$\\
\hline
$X^{(1)}$&$X^{(1)}$&$X^{(3)}$&$X^{(5)}$&$X^{(6)}$&$X^{(7)}$&$X^{(9)}$&$X^{(10)}$&$X^{(11)}$&$X^{(12)}$&$X^{(13)}$ \\
\hline
$X^{(3)}$& &(r1)&$sX^{(8)}$&0&(r2)&0&$-sX^{(14)}$&0&0&$X^{(15)}$ \\
\hline
$X^{(5)}$& & &0&0&0&0&0&0&$X^{(16)}$&$X^{(17)}$ \\
\hline
$X^{(6)}$& & & &0&0&0&$-X^{(16)}$&0&0&$X^{(19)}$\\
\hline
$X^{(7)}$& & & & &0&$sX^{(16)}$&$-X^{(17)}$&0&$X^{(19)}$&(r3) \\
\hline
$X^{(9)}$& & & & & &0&0&0&0&$sX^{(22)}$\\
\hline
$X^{(10)}$& & & & & & &0&0&$-X^{(22)}$&$-X^{(23)}$\\
\hline
$X^{(11)}$& & & & & & & &0&0&0\\
\hline
$X^{(12)}$& & & & & & & & &0&$X^{(24)}$\\
\hline
$X^{(13)}$& & & & & & & & & &$-X^{(25)}$\\
\end{tabular}

$\quad$

$\quad$

\begin{tabular}{c|c|c|c|c|c|c|c|c|c|c}
&$X^{(15)}$&$X^{(17)}$&$X^{(18)}$&$X^{(19)}$&$X^{(20)}$&$X^{(21)}$&$X^{(23)}$&$X^{(24)}$&$X^{(25)}$&$X^{(27)}$\\
\hline
$X^{(1)}$&$X^{(15)}$&$X^{(17)}$&$X^{(18)}$&$X^{(19)}$&$X^{(20)}$&$X^{(21)}$&$X^{(23)}$&$X^{(24)}$&$X^{(25)}$&$X^{(27)}$ \\
\hline
$X^{(3)}$&(r4)&0&$-2sX^{(22)}$&0&0&(r5)&$-sX^{(26)}$&0&$X^{(27)}$&0 \\
\hline
$X^{(5)}$&0&0&0&0&0&$-sX^{(26)}$&0&$-X^{(28)}$&0&0 \\
\hline
$X^{(6)}$&0&0&0&0&0&0&$-X^{(28)}$&0&0&0\\
\hline
$X^{(7)}$&0&0&$2sX^{(26)}$&0&0&$-X^{(27)}$&0&0&0&0\\
\hline
$X^{(9)}$&0&0&0&0&$-X^{(28)}$&0&0&0&$-sX^{(4)}$&0\\
\hline
$X^{(10)}$&$sX^{(26)}$&0&0&$X^{(28)}$&0&$-sX^{(2)}$&0&$-X^{(4)}$&$X^{(5)}$&$sX^{(8)}$\\
\hline
$X^{(11)}$&0&0&$-X^{(28)}$&0&0&0&0&0&0&0\\
\hline
$X^{(12)}$&0&$-X^{(28)}$&0&0&0&0&$X^{(4)}$&0&$-X^{(6)}$&0\\
\hline
$X^{(13)}$&$-X^{(27)}$&0&$-2sX^{(2)}$&0&0&$X^{(3)}$&$X^{(5)}$&$X^{(6)}$&$-X^{(7)}$&(r6)\\
\end{tabular}

$\quad$

$\quad$

\begin{tabular}{c|c|c|c|c|c|c|c|c|c|c}
&$X^{(15)}$&$X^{(17)}$&$X^{(18)}$&$X^{(19)}$&$X^{(20)}$&$X^{(21)}$&$X^{(23)}$&$X^{(24)}$&$X^{(25)}$&$X^{(27)}$\\
\hline
$X^{(15)}$&0&0&$-2sX^{(4)}$&0&0&(r7)&$sX^{(8)}$&0&(r8)&0 \\
\hline
$X^{(17)}$& &0&0&0&0&$sX^{(8)}$&0&0&0&0 \\
\hline
$X^{(18)}$& & &0&0&$-2X^{(8)}$&$-2X^{(9)}$&0&0&$2sX^{(14)}$&$2sX^{(16)}$\\
\hline
$X^{(19)}$& & & &0&0&0&0&0&0&0 \\
\hline
$X^{(20)}$& & & & &0&$-2X^{(11)}$&0&0&0&0 \\
\hline
$X^{(21)}$& & & & & &(r9)&$sX^{(14)}$&0&$-X^{(15)}$&(r10)\\
\hline
$X^{(23)}$& & & & & & &0&$X^{(16)}$&$-X^{(17)}$&0 \\
\hline
$X^{(24)}$& & & & & & & &0&$-X^{(19)}$&0 \\
\hline
$X^{(25)}$& & & & & & & & &(r11)&0 \\
\hline
$X^{(27)}$& & & & & & & & & &0 \\
\end{tabular}

\begin{thm}\label{main_thm8}
Let $s>1$, and let $R=R^\prime_s$ be an algebra of type $E_8$. Then the Hochschild cohomology ring $\HH^*(R)$ is
isomorphic to $\mathcal A$ as a graded $K$-algebra.
\end{thm}

Consider the case of $s=1$.

Let $\mathcal X^\prime=\mathcal X\cup \left\{X^{(29)}_0, X^{(30)}_0, X^{(31)}_0, 
X^{(32)}_0, X^{(33)}_0, X^{(34)}_0, X^{(35)}_0, X^{(36)}_0\right\}.$

We define a grading of the ring $K[\mathcal X^\prime]$, such that
\begin{align*}
&\deg X^{(i)}_{t_{j,i}}=t_{j, i} \:\text{for all} \: i=1,\dots,28 \:\text{and}\: j=1,\dots,\alpha_i;\\
&\deg T=M \text{ (similar to (\ref{degs28}))};\\&\deg X^{(i)}_0=0 \:\text{for all} \: i=29,\dots,36.
\end{align*}

Then $\mathcal A^\prime=K[\mathcal X^\prime]/I^\prime$ is a graded $K$-algebra, where $I^\prime$ 
is the ideal generated by the homogeneous elements corresponding to the relations described for $s>1$, 
and the following relations:
\begin{align*}
X^{(1)}X^{(i)}=&\begin{cases}\widetilde X^{(i)},\quad
t_1=0;\\0,\quad\text{otherwise};\end{cases}\\
X^{(j)}X^{(i)}=&0,\quad j\in[2, 36],\quad i\in[29, 36],
\end{align*}
where $t_1$ is the degree of $X^{(1)}$.

\begin{thm}\label{main_thm82}
Let $s=1$, and let $R=R^\prime_1$ be an algebra of type $E_8$. Then the Hochschild cohomology ring $\HH^*(R)$ is
isomorphic to $\mathcal A^\prime$ as a graded $K$-algebra.
\end{thm}

\begin{zam}
From the description of rings in Theorems
\ref{main_thm8} and \ref{main_thm82}, it follows that they are commutative.
\end{zam}

\section{Tree class $E_8$: Bimodule resolution}\label{sect_res8}

We are going to construct the minimal projective bimodule resolution of $R$ in the form: $$
\dots\longrightarrow Q_3\stackrel{d_2}\longrightarrow Q_2\stackrel{d_1}\longrightarrow
Q_1\stackrel{d_0}\longrightarrow Q_0\stackrel\varepsilon\longrightarrow R\longrightarrow 0
$$

Let $\Lambda$ be the enveloping algebra of $R$. Then $R$--$R$-bimodules can be considered as left $\Lambda$-modules.

\begin{obozns}$\quad$

(1) Let $e_i,\text{ }i\in \Z_{8s}=\{0, 1,\dots, 8s-1\},$ be the idempotents of the algebra
$K\left[\mathcal Q^\prime_s\right]$, corresponding to the vertices of the quiver $\mathcal Q^\prime_s$.

(2) Set $P_{i,j}=R(e_i\otimes e_j)R=\Lambda(e_i\otimes e_j)$, $i,j\in \Z_{8s}$. 
We note that the modules $P_{i,j}$ form a full set of (pairwise non-isomorphic) 
indecomposable projective $\Lambda$-modules.

(3) For $a\in\Z$, $t\in\N$, set $(a)_t$ to be the smallest nonnegative residue of $a$ modulo $t$
 (in particular, $0\le(a)_t\le t-1$).

\end{obozns}

Let $R=R^\prime_s$. We define an automorphism $\rho\text{: }R\rightarrow R$, by the following formulas:
$$\rho(e_i)=e_{i+15n},\quad\rho(\g_i)=-\g_{i+15},$$
 $$\rho(\a_i)=\begin{cases}
 \a_{i+15\cdot 5},\quad (i)_5=4;\\
 -\a_{i+15\cdot 5},\quad (i)_5\ne 4,\\
 \end{cases}\quad\rho(\b_i)=\begin{cases}
 \b_{i+15\cdot 3},\quad (i)_3=0;\\
 -\b_{i+15\cdot 3},\quad (i)_3\ne 0.\end{cases}$$

Let $f\text{: }\Z\times\Z\rightarrow\Z$ and $f\text{: }\Z\times\Z\times\Z\rightarrow\Z$
be auxiliary functions defined as follows:
$$f(x,y)=\begin{cases}1,\quad x=y;\\0,\quad x\ne y,\end{cases}\quad
f(x,y_1,y_2)=\begin{cases}1,\quad y_1\le x\le y_2;\\0,\quad\text{otherwise.}
\end{cases}$$

Recall that for the degree $r$, $m$ is the integral part of $r$ modulo $2$. Let us define
$Q_r\text{ }(r\le 28)$ as follows:
\begin{align*}
Q_{2m}&=\bigoplus_{r=0}^{s-1} Q_{2m,r}^\prime,\quad 0\le m\le 14,\\
Q_{2m+1}&=\bigoplus_{r=0}^{s-1} Q_{2m+1,r}^\prime,\quad 0\le m\le 13.
\end{align*}
\begin{multline*}
Q_{2m,r}^\prime=\left(\bigoplus_{i=0}^{f(m,2,11)+f(m,4,9)}P_{b_0(r,m,i),8r}\right)
\oplus\bigoplus_{j=0}^3\bigoplus_{i=0}^{f(m+j,5)+f(m+j,7,10)+f(m+j,12)}P_{b_1(r,m,i,j),8r+j+1}\\
\oplus\bigoplus_{j=0}^1\bigoplus_{i=0}^{f(m+j,3,12)+f(m+j,5,10)}P_{b_2(r,m,i,j),8r+j+5}
\oplus\left(\bigoplus_{i=0}^{f(m,3,12)+f(m,5,10)}P_{b_3(r,m,i),8r+7}\right),
\end{multline*} 
where
$$\begin{aligned}
b_0(r,m,i)=&8(r+m)-f(i,0)(f(m,1,10)+f(m,13))\\
&-f(i,1)(f(m,4)+f(m,11))-f(i,2)(f(m,6)+f(m,8));\\
\quad b_1(r,m,i,j)=&8(r+m)+m+j+1-f(i,0)(5f(m+j,6,9)+9f(m+j,10,12)\\
&+8f(m+j,13)+14f(m+j,14,17))\\
&-f(i,1)(5f(m+j,5)+3f(m+j,7,8)+9f(m+j,9)\\
&+5f(m+j,10)+8f(m+j,12));\\
b_2(r,m,i,j)=&8(r+m)+m+j-1+6f(m+j,0,1)\\
&+f(i,0)(f(m+j,5)-3f(m+j,6)-5f(m+j,8,9)\\&-3f(m+j,10)-8f(m+j,11,15))\\
&+f(i,1)(3f(m+j,3,4)+f(m+j,6)-3f(m+j,7,8)\\&-2f(m+j,9)-5f(m+j,10,12))\\
&-f(i,2)(3f(m+j,5)+5f(m+j,7)+2f(m+j,8)\\&+3f(m+j,9)+8f(m+j,10));\\
b_3(r,m,i)=&8(r+m+1)-f(i,0)(f(m,0)+f(m,2,11)+f(m,14))\\
&-f(i,1)(f(m,5)+f(m,12))-f(i,2)(f(m,7)+f(m,9)).
\end{aligned}$$

\begin{multline*}
Q_{2m+1,r}^\prime=\left(\bigoplus_{i=0}^{f(m,0,12)+f(m,2,10)+f(m,3,9)+f(m,5)+f(m,7)}P_{b_4(r,m,i),8r}\right)
\oplus\bigoplus_{j=0}^3\bigoplus_{i=0}^{f(m+j,8)}P_{b_5(r,m,i,j),8r+j+1}\\
\oplus\bigoplus_{j=0}^1\bigoplus_{i=0}^{f(m+j,4,10)}P_{b_6(r,m,i,j),7r+j+5}
\oplus\left(\bigoplus_{i=0}^{f(m,1,13)+f(m,3,11)+f(m,4,10)+f(m,6)+f(m,8)}P_{b_7(r,m,i),8r+7}\right),
\end{multline*}
where
$$\begin{aligned}
b_4(r,m,i)=&8(r+m)+f(i,0)(2-f(m,0)+3f(m,2)+3f(m,4)+f(m,7)+f(m,9)\\
&+f(m,11)+2f(m,12)+5f(m,13)+6f(m,14))\\
&+f(i,1)(6-f(m,0)-3f(m,2)-2f(m,3)-3f(m,4)\\
&-3f(m,6)-2f(m,8)-2f(m,10)-f(m,11))\\
&+f(i,2)(6-5f(m,2)-2f(m,5)-2f(m,7)-5f(m,9))\\
&+f(i,3)(5-4f(m,4,5)-4f(m,7))+5f(i,4);\\
b_5(r,m,i,j)=&8(r+m)+7-5f(m+j,0)-4f(m+j,1)-3f(m+j,2)+f(m+j,4)\\
&+f(m+j,6)+f(m+j,9)+f(m+j,11)+f(m+j,13)+2f(m+j,14)\\
&+3f(m+j,15)+4f(m+j,16)+f(i,1)f(m+j,8);\\
b_6(r,m,i,j)=&8(r+m)+7-f(m+j,0)+f(m+j,2)+f(m+j,8)+f(m+j,11)\\
&+f(m+j,13)+6f(m+j,14)+f(i,1)(1-2f(m+j,8));\\
b_7(r,m,i)=&8(r+m+1)+f(i,0)(2-2f(m,0)-f(m,1)+3f(m,3)+3f(m,5)\\
&+f(m,8)+f(m,10)+f(m,12)+2f(m,13)+5f(m,14))\\
&+f(i,1)(6-f(m,1)-3f(m,3)-2f(m,4)-3f(m,5)\\
&-3f(m,7)-2f(m,9)-2f(m,11)-f(m,12))\\
&+f(i,2)(6-5f(m,3)-2f(m,6)-2f(m,8)-5f(m,10))\\
&+f(i,3)(5-4f(m,5,6)-4f(m,8))+5f(i,4).
\end{aligned}$$

\begin{thm}\label{resol_thm8}
Let $R=R^\prime_s$ is algebra of the type $E_8$. Then the minimal projective resolution of the
$\Lambda$-module $R$ is of the form:
\begin{equation}\label{resolv8}\tag{$++$} \dots\longrightarrow
Q_3\stackrel{d_2}\longrightarrow Q_2\stackrel{d_1}\longrightarrow
Q_1\stackrel{d_0}\longrightarrow
Q_0\stackrel\varepsilon\longrightarrow R\longrightarrow
0,
\end{equation}
where $\varepsilon$ is the multiplication map $(\varepsilon(a\otimes b)=ab)$; $Q_r\text{ }(r\le
28)$ are as above, $d_r\text{ }(r\le 28)$ are described in ancillary files of this paper; further $Q_{29\ell+r}$, where $\ell\in \N$
and $0\le r\le 28$, is obtained from $Q_r$ by replacing each direct summand $P_{i,j}$ with
$P_{\sigma^{\prime\ell}(i),j}$ respectively $($here $\rho(i)=j$, if $\rho(e_i)=e_j)$, and the
differential $d_{29\ell+r}$ is obtained from $d_r$ by applying $\sigma^{\prime\ell}$ to all left tensor
components of the corresponding matrix.
\end{thm}

To prove that the $Q_i$ are of this form we introduce the projective cover $P_i=Re_i$ of
the simple $R$-modules $S_i$ corresponding to the vertices of the quiver $\mathcal Q^\prime_s$. Let us
find projective resolutions of the simple $R$-modules $S_i$.

\begin{obozn}
The $m$th syzygy of an $R$-module $M$ is denoted by $\Omega^m(M)$.
\end{obozn}

\begin{zam}
In what follows, the homomorphism induced by the right multiplication by an
element $w$ is denoted by $w$.
\end{zam}

\begin{lem}\label{lem_s08}
The beginning of the minimal projective resolution of $S_{8r}$ is of the form

\begin{multline*}
\dots\longrightarrow P_{8(r+12)+7}\xrightarrow{\binom{\a}{-\b}}
P_{8(r+12)+4}\oplus P_{8(r+m)+6}\xrightarrow{(\a^{4}\text{ }\b^{2})}
P_{8(r+12)}\xrightarrow{\binom{\g\a^{2}}{-\g\b^{2}}}\\
\longrightarrow
P_{8(r+11)+3}\oplus P_{8(r+11)+5}
\xrightarrow{\binom{\phantom{-}\a^{3}\phantom{-:}\b}{-\a^{3}\g\phantom{-}0}} 
P_{8(r+11)}\oplus P_{8(r+10)+7}
\xrightarrow{\triplet{\g\a^{3}\phantom{-}0}{\g\a\phantom{:-}\a}{\phantom{:}0\phantom{-}-\b}}\\
\longrightarrow 
P_{8(r+10)+2}\oplus P_{8(r+10)+4}\oplus P_{8(r+10)+6}
\xrightarrow{\binom{\a^{2}\g\phantom{-:}0\phantom{::-}0}{-\a^{2}\phantom{-}\a^{4}\phantom{-}\b^{2}}}
P_{8(r+9)+7}\oplus P_{8(r+10)}
\xrightarrow{\quatro{\phantom{::}\a^{2}\phantom{::-}\g\a^{2}}{-\b\phantom{--::}0\phantom{:}}
{-\a^{4}\phantom{--:}0\phantom{-}}{-\b^{2}\phantom{::}-\g\b^{2}}}\\
\longrightarrow
P_{8(r+9)+3}\oplus P_{8(r+9)+6}\oplus P_{8(r+9)+1}\oplus P_{8(r+9)+5}
\xrightarrow{\triplet{\a^{3}\g\phantom{::}\b^{2}\g\phantom{::}0\phantom{-}0}
{\phantom{:}\a^{3}\phantom{-::}0\phantom{::-}0\phantom{-}\b}
{\phantom{::}0\phantom{:-}-\b^{2}\phantom{::}\a\phantom{-}0}}\\
\longrightarrow
P_{8(r+8)+7}\oplus P_{8(r+9)}\oplus P_{8(r+9)}
\xrightarrow{\quatro{\phantom{:}\a^{3}\phantom{-}0\phantom{--}0}{-\a\phantom{::}\g\a\phantom{-:}0}
{\phantom{:}\b\phantom{--}0\phantom{:-}\g\b}{\b^{2}\phantom{-:}0\phantom{--}0}}\\
\longrightarrow 
P_{8(r+8)+2}\oplus P_{8(r+8)+4}\oplus P_{8(r+8)+6}\oplus P_{8(r+8)+5}
\xrightarrow{\triplet{\a^{2}\g\phantom{::}\a^{4}\g\phantom{::}0\phantom{:-}0}
{\phantom{:}-\a^{2}\phantom{-}0\phantom{-:}0\phantom{:-}\b}
{-\a^{2}\g\phantom{::}0\phantom{::}\b^{2}\g\phantom{::}0}}\\
\longrightarrow
P_{8(r+7)+7}\oplus P_{8(r+8)}\oplus P_{8(r+7)+7}
\xrightarrow{\quintet{\a^{2}\phantom{-}\g\a^{2}\phantom{::}0\phantom{::}}
{-\b\phantom{-:}0\phantom{-::}0\phantom{::}}{0\phantom{-}-\g\a\phantom{::}\a}
{-\a^{4}\phantom{::}0\phantom{-::}0\phantom{-}}{-\b^{2}\phantom{::}0\phantom{::}-\b^{2}}}\\
\longrightarrow 
P_{8(r+7)+3}\oplus P_{8(r+7)+6}\oplus P_{8(r+7)+4}\oplus P_{8(r+7)+1}\oplus P_{8(r+7)+5}
\xrightarrow{\triplet{\a^{3}\g\phantom{::}\b^{2}\g\phantom{::}0\phantom{-}0\phantom{-}0}
{\phantom{::}\a^{3}\phantom{-:}0\phantom{::-}\a^{4}\phantom{:}0\phantom{-}\b}
{\phantom{::}0\phantom{::-}-\b^{2}\phantom{::}0\phantom{::}\a\phantom{-}0}}\\
\longrightarrow 
P_{8(r+6)+7}\oplus P_{8(r+7)}\oplus P_{8(r+7)}
\xrightarrow{\quatro{\phantom{::}\a^{3}\phantom{:-}0\phantom{--}0}{-\a^{2}\phantom{::}\g\a^{2}\phantom{::}0}
{\phantom{::}\b\phantom{--}0\phantom{:-}\g\b}{\phantom{:}\b^{2}\phantom{-:}0\phantom{--}0}}\\
\longrightarrow
P_{8(r+6)+2}\oplus P_{8(r+6)+3}\oplus P_{8(r+6)+6}\oplus P_{8(r+6)+5}
\xrightarrow{\triplet{\a^{2}\g\phantom{::}\a^{3}\g\phantom{::}0\phantom{:-}0}
{\phantom{:}-\a^{2}\phantom{-}0\phantom{-:}0\phantom{:-}\b}
{-\a^{2}\g\phantom{::}0\phantom{::}\b^{2}\g\phantom{::}0}}\\
\longrightarrow
P_{8(r+5)+7}\oplus P_{8(r+6)}\oplus P_{8(r+5)+7}
\xrightarrow{\quintet{\a^{3}\phantom{-}\g\a^{3}\phantom{::}0\phantom{::}}
{-\b\phantom{-:}0\phantom{-::}0\phantom{::}}{0\phantom{-}-\g\a\phantom{::}\a}
{-\a^{4}\phantom{::}0\phantom{-::}0\phantom{-}}{-\b^{2}\phantom{::}0\phantom{::}-\b^{2}}}\\
\longrightarrow
P_{8(r+5)+2}\oplus P_{8(r+5)+6}\oplus P_{8(r+5)+4}\oplus P_{8(r+5)+1}\oplus P_{8(r+5)+5}
\xrightarrow{\triplet{\a^{2}\g\phantom{::}\b^{2}\g\phantom{::}0\phantom{-}0\phantom{-}0}
{\phantom{::}\a^{2}\phantom{-:}0\phantom{::-}\a^{4}\phantom{:}0\phantom{-}\b}
{\phantom{::}0\phantom{::-}-\b^{2}\phantom{::}0\phantom{::}\a\phantom{-}0}}\\
\longrightarrow
P_{8(r+4)+7}\oplus P_{8(r+5)}\oplus P_{8(r+5)}
\xrightarrow{\quatro{\phantom{:}\b^{2}\phantom{-::}0\phantom{--}0}{-\a^{2}\phantom{::}\g\a^{2}\phantom{::}0}
{\phantom{-}\b\phantom{--}0\phantom{-}\g\b}{-\a^{4}\phantom{-}0\phantom{-::}0}}\\
\longrightarrow
P_{8(r+4)+5}\oplus P_{8(r+4)+3}\oplus P_{8(r+4)+6}\oplus P_{8(r+4)+1}
\xrightarrow{\triplet{\b\g\phantom{-}\a^{3}\g\phantom{-}0\phantom{-:}0}
{-\b\g\phantom{-}0\phantom{-}\b^{2}\g\phantom{-}0}{\phantom{-}\b\phantom{-::}0\phantom{-:::}0\phantom{:-}\a}}\\
\longrightarrow
P_{8(r+3)+7}\oplus P_{8(r+3)+7}\oplus P_{8(r+4)}
\xrightarrow{\quatro{\phantom{:}\a^{3}\phantom{-:}0\phantom{--:}0}{\phantom{:}0\phantom{--}\a\phantom{--:}0}
{-\b\phantom{--}0\phantom{--}\g\b}{-\b^{2}\phantom{-}-\b^{2}\phantom{-}0}}
\end{multline*}
\begin{multline*}
\longrightarrow
P_{8(r+3)+2}\oplus P_{8(r+3)+4}\oplus P_{8(r+3)+6}\oplus P_{8(r+3)+5}
\xrightarrow{\binom{\a^{2}\g\phantom{-}0\phantom{-}\b^{2}\g\phantom{-}0}
{\phantom{-}\a^{2}\phantom{-}\a^{4}\phantom{-}0\phantom{--}\b}}\\
\longrightarrow
P_{8(r+2)+7}\oplus P_{8(r+3)}\xrightarrow{\triplet{\phantom{-}\b^{2}\phantom{:-}0\phantom{:}}
{-\a^{2}\phantom{::}\g\a^{2}}{-\a^{4}\phantom{:-}0\phantom{:}}}
P_{8(r+2)+5}\oplus P_{8(r+2)+3}\oplus P_{8(r+2)+1}\xrightarrow{\binom{\b\g\phantom{-}\a^{3}\g\phantom{-}0}{\phantom{-}\b\phantom{:-}\phantom{-}0\phantom{:-}\a}}\\
\longrightarrow
P_{8(r+1)+7}\oplus P_{8(r+2)}\xrightarrow{\binom{\phantom{-}\a^{3}\phantom{-}0}{-\b\phantom{-}\g\b}}
P_{8(r+1)+2}\oplus P_{8(r+1)+6}\xrightarrow{(\a^{2}\g\text{ }\b^{2}\g)}\\
\longrightarrow 
P_{8r+7}\xrightarrow{\binom{\a^{4}}{-\b^{2}}} 
P_{8r+1}\oplus P_{8r+5}\xrightarrow{(\a\text{ }\b)} 
P_{8r}\longrightarrow S_{8r}\longrightarrow 0.
\end{multline*}

Moreover, $\Omega^{27}(S_{8r})\simeq S_{8(r+13)+7}$.
\end{lem}

\begin{lem}
The beginning of the minimal projective resolution of $S_{8r+1}$ is of the form $$\dots\longrightarrow
P_{8r+2} \stackrel{\a}\longrightarrow P_{8r+1}\longrightarrow S_{8r+1}\longrightarrow 0.$$ Moreover,
$\Omega^{2}(S_{8r+1})\simeq S_{8(r+1)+2}$.
\end{lem}

\begin{lem}
The beginning of the minimal projective resolution of $S_{8r+2}$ is of the form $$\dots\longrightarrow
P_{8r+3} \stackrel{\a}\longrightarrow P_{8r+2}\longrightarrow S_{8r+2}\longrightarrow 0.$$ Moreover,
$\Omega^{2}(S_{8r+2})\simeq S_{8(r+1)+3}$.
\end{lem}

\begin{lem}
The beginning of the minimal projective resolution of $S_{8r+3}$ is of the form $$\dots\longrightarrow
P_{8r+4} \stackrel{\a}\longrightarrow P_{8r+3}\longrightarrow S_{8r+3}\longrightarrow 0.$$ Moreover,
$\Omega^{2}(S_{8r+3})\simeq S_{8(r+1)+4}$.
\end{lem}

\begin{lem}
The beginning of the minimal projective resolution of $S_{8r+4}$ is of the form
\begin{multline*}
\dots\longrightarrow
P_{8(r+11)+1}\stackrel{\a}\longrightarrow
P_{8(r+11)}\stackrel{\g\b}\longrightarrow
P_{8(r+10)+6}\stackrel{\b^{2}\g}\longrightarrow
P_{8(r+9)+7}\xrightarrow{\binom{\a}{-\b^{2}}}\\
\longrightarrow
P_{8(r+9)+4}\oplus P_{8(r+9)+5}\xrightarrow{(\a^{4}\text{ }\b)}
P_{8(r+9)}\xrightarrow{\g\a^{2}}
P_{8(r+8)+3}\xrightarrow{\a^{3}\g}
P_{8(r+7)+7}\xrightarrow{\binom{\a^{3}}{-\b}}\\
\longrightarrow
P_{8(r+7)+2}\oplus P_{8(r+7)+6}\xrightarrow{(\a^{2}\text{ }\b^{2})}
P_{8(r+7)}\xrightarrow{\binom{\g\b^{2}}{-\g\a^{4}}}
P_{8(r+6)+5}\oplus P_{8(r+6)+1}\xrightarrow{\binom{\b\g\phantom{::}0}{\b\phantom{-}\a}}\\
\longrightarrow
P_{8(r+5)+7}\oplus P_{8(r+6)}\xrightarrow{\binom{\phantom{:}\a\phantom{-}0}{-\b\phantom{-}\g\b}}
P_{8(r+5)+4}\oplus P_{8(r+5)+6}\xrightarrow{(\a^{4}\g\text{ }\b^{2}\g)}
P_{8(r+4)+7}\xrightarrow{\binom{\a^{2}}{-\b^{2}}}\\
\longrightarrow
P_{8(r+4)+3}\oplus P_{8(r+4)+5}\xrightarrow{(\a^{3}\text{ }\b)}
P_{8(r+4)}\xrightarrow{\g\a^{3}}
P_{8(r+3)+2}\xrightarrow{\a^{2}\g}
P_{8(r+2)+7}\xrightarrow{\binom{\b}{-\a^{4}}}\\
\longrightarrow
P_{8(r+2)+6}\oplus P_{8(r+2)+1}\xrightarrow{(\b^{2}\text{ }\a)}
P_{8(r+2)}\xrightarrow{\g\b^{2}}\\
\longrightarrow
P_{8(r+1)+5}\stackrel{\b\g}\longrightarrow P_{8r+7}
\stackrel{\a}\longrightarrow P_{8r+4}\longrightarrow S_{8r+4}\longrightarrow 0.
\end{multline*}

Moreover, $\Omega^{23}(S_{8r+4})\simeq S_{8(r+12)+1}$.
\end{lem}

\begin{lem}
The beginning of the minimal projective resolution of $S_{8r+5}$ is of the form $$\dots\longrightarrow
P_{8r+6} \stackrel{\b}\longrightarrow P_{8r+5}\longrightarrow S_{8r+5}\longrightarrow 0.$$ Moreover,
$\Omega^{2}(S_{8r+5})\simeq S_{8(r+1)+6}$.
\end{lem}

\begin{lem}
The beginning of the minimal projective resolution of $S_{8r+6}$ is of the form
\begin{multline*}
\dots\longrightarrow
P_{8(r+13)+5}\stackrel{\b}\longrightarrow
P_{8(r+13)}\stackrel{\g\a}\longrightarrow
P_{8(r+12)+4}\xrightarrow{\a^{4}\g}\\
\longrightarrow
P_{8(r+11)+7}\xrightarrow{\binom{\a^{2}}{-\b}}
P_{8(r+11)+3}\oplus P_{8(r+11)+6}\xrightarrow{(\a^{3}\text{ }\b^{2})}
P_{8(r+11)}\xrightarrow{\binom{\g\a^{3}}{-\g\b^{2}}}\\
\longrightarrow
P_{8(r+10)+2}\oplus P_{8(r+10)+5}
\xrightarrow{\binom{\a^{2}\g\phantom{-}0}{\phantom{:}\a^{2}\phantom{-:}\b}}
P_{8(r+9)+7}\oplus P_{8(r+10)}
\xrightarrow{\triplet{\phantom{::}\b\phantom{-::}0}{-\a\phantom{-}\g\a}{-\a^{4}\phantom{-}0}}\\
\longrightarrow
P_{8(r+9)+6}\oplus P_{8(r+9)+4}\oplus P_{8(r+9)+1}
\xrightarrow{\binom{\b^{2}\g\phantom{::}\a^{4}\g\phantom{::}0}
{\phantom{-}\b^{2}\phantom{::-}0\phantom{:-}\a}}
P_{8(r+8)+7}\oplus P_{8(r+9)}
\xrightarrow{\triplet{\a^{2}\phantom{-}0}{-\b\phantom{-}\g\b}{-\b^{2}\phantom{-}0\phantom{:}}}\\
\longrightarrow
P_{8(r+8)+3}\oplus P_{8(r+8)+6}\oplus P_{8(r+8)+5}
\xrightarrow{\binom{\phantom{::}\a^{3}\phantom{-:}0\phantom{-:}\b}{\a^{3}\g\phantom{::}\b^{2}\g\phantom{::}0}}
P_{8(r+8)}\oplus P_{8(r+7)+7}
\xrightarrow{\triplet{\phantom{:}\g\a^{3}\phantom{::}0}{-\g\a\phantom{-}\a}{\phantom{-:}0\phantom{-}-\b^{2}}}\\
\longrightarrow
P_{8(r+7)+2}\oplus P_{8(r+7)+4}\oplus P_{8(r+7)+5}
\xrightarrow{\binom{\a^{2}\g\phantom{-}0\phantom{:-}0}{\phantom{:}\a^{2}\phantom{-}\a^{4}\phantom{-}\b}}
P_{8(r+6)+7}\oplus P_{8(r+7)}
\xrightarrow{\triplet{\b\phantom{--}0}{-\a^{2}\phantom{::}\g\a^{2}}{-\a^{4}\phantom{-::}0\phantom{::}}}\\
\longrightarrow
P_{8(r+6)+6}\oplus P_{8(r+6)+3}\oplus P_{8(r+6)+1}
\xrightarrow{\binom{\b^{2}\g\phantom{-}\a^{3}\g\phantom{-}0}{\phantom{-}\b^{2}\phantom{--}0\phantom{:-}\a}}
P_{8(r+5)+7}\oplus P_{8(r+6)}\xrightarrow{\triplet{\phantom{-}\a^{3}\phantom{-}0}{-\b\phantom{-}\g\b}{-\b^{2}\phantom{-}0}}\\
\longrightarrow
P_{8(r+5)+2}\oplus P_{8(r+5)+6}\oplus P_{8(r+5)+5}\xrightarrow{\binom{\a^{2}\g\phantom{-}\b^{2}\g\phantom{-}0}{\phantom{-}\a^{2}\phantom{-}\phantom{-}0\phantom{-}\b}}
P_{8(r+4)+7}\oplus P_{8(r+5)}
\xrightarrow{\triplet{\b^{2}\phantom{-}0}{-\a\phantom{-}\g\a}{-\a^{4}\phantom{-}0\phantom{:}}}\\
\longrightarrow
P_{8(r+4)+5}\oplus P_{8(r+4)+4}\oplus P_{8(r+4)+1}
\xrightarrow{\binom{\b\g\phantom{-}\a^{4}\g\phantom{-}0}{\b\phantom{--}0\phantom{--}\a}}
P_{8(r+3)+7}\oplus P_{8(r+4)}\xrightarrow{\binom{\phantom{:}\a^{2}\phantom{-}0}{-\b\phantom{-}\g\b}}\\
\longrightarrow
P_{8(r+3)+3}\oplus P_{8(r+3)+6}\xrightarrow{(\a^{3}\g\text{ }\b^{2}\g)}
P_{8(r+2)+7}\xrightarrow{\binom{\a^{3}}{-\b^{2}}}
P_{8(r+2)+2}\oplus P_{8(r+2)+5}\xrightarrow{(\a^{2}\text{ }\b)}\\
\longrightarrow
P_{8(r+2)}\stackrel{\g\a^{4}}\longrightarrow
P_{8(r+1)+1}\stackrel{\a\g}\longrightarrow
P_{8r+7}\stackrel{\b}\longrightarrow P_{8r+6}\longrightarrow S_{8r+6}\longrightarrow 0.
\end{multline*}
Moreover, $\Omega^{27}(S_{8r+6})\simeq S_{8(r+14)+5}$.
\end{lem}

\begin{lem}\label{lem_s78}
The beginning of the minimal projective resolution of $S_{8r+7}$ is of the form $$\dots\longrightarrow
P_{8(r+1)} \stackrel{\g}\longrightarrow P_{8r+7}\longrightarrow S_{8r+7}\longrightarrow 0.$$ Moreover,
$\Omega^{2}(S_{8r+7})\simeq S_{8(r+2)}$.
\end{lem}

\begin{proof}
The proofs of the lemmas are easily obtained by a direct verification that a given
sequence is exact, and are straightforward.
\end{proof}

\begin{proof}[Proof of the theorem \ref{resol_thm8}]
The fact that the $Q_i$ have a desired form immediately follows from Lemmas 
\ref{lem_s08} -- \ref{lem_s78} and Happel's
lemma.

As was shown in \cite{VGI}, to prove that sequence \eqref{resolv8} is exact in $Q_m$ ($m\le 29$), 
it suffices to verify that $d_md_{m+1}=0$. These relations are verified by a straightforward 
calculation of the products of the corresponding matrices.

Since the sequence is exact in $Q_{29}$, it follows that
$\Omega^{29}({}_\Lambda R)\simeq {}_1R_{\rho}$, where
$\Omega^{29}({}_\Lambda R)=\Im d_{28}$ is the 29th syzygy of the
module $R$, and ${}_1R_{\rho}$ is a twisted bimodule. Hence, the
exactness in $Q_t$ ($t>29$) holds.
\end{proof}

\begin{s}
$\Omega^{29}({}_\Lambda R)\simeq {}_1R_{\rho}$.
\end{s}

\begin{pr}
The automorphism $\rho$ has a finite order and

$(1)$ if $\myChar=2$, then the order of $\rho$ is equal to $\frac{s}{\myNod(s,15)}$;

$(2)$ if $\myChar\ne 2$, then the order of $\rho$ is equal to $\frac{s}{\myNod(s,15)}$, if
$\frac{s}{\myNod(s,15)}$ is even, and $\frac{2s}{\myNod(s,15)}$ otherwise.
\end{pr}

\begin{pr}
The minimal period of the bimodule resolution of $R$ is $29\ord\rho$.
\end{pr}

\section{Tree class $E_8$: The additive structure of $\HH^*(R)$}

\begin{pr}[Dimensions of homomorphism groups, $s>1$]\label{dim_hom8}
Let $s>1$ and $R=R^\prime_s$ be an algebra of the $E_8$. Next, let $deg\in\N\cup\{0\}$,
and $\ell$ and $r$ the integral part and residue of $deg$ modulo $29$. Then

$(1)$ If $r\in\{0,14,28\}$, then $$\dim_K\Hom_\Lambda(Q_{deg}, R)=\begin{cases}
 8s,\quad m+15\ell\equiv 0(s)\text{ or }m+15\ell\equiv 1(s);\\
 0,\quad\text{otherwise.}
 \end{cases}$$

$(2)$ If $r\in\{1,27\}$, then $$\dim_K\Hom_\Lambda(Q_{deg}, R)=\begin{cases}
 9s,\quad m+15\ell\equiv 0(s);\\
 0,\quad\text{otherwise.}
 \end{cases}$$

$(3)$ If $r=2$, then $$\dim_K\Hom_\Lambda(Q_{deg},
R)=\begin{cases}
 5s,\quad m+15\ell\equiv 0(s);\\
 s,\quad m+15\ell\equiv 1(s);\\
 0,\quad\text{otherwise.}
 \end{cases}$$

$(4)$ If $r\in\{3,25\}$, then $$\dim_K\Hom_\Lambda(Q_{deg},
R)=\begin{cases}
 10s,\quad m+15\ell\equiv 0(s);\\
 0,\quad\text{otherwise.}
 \end{cases}$$

$(5)$ If $r=4$, then $$\dim_K\Hom_\Lambda(Q_{deg},
R)=\begin{cases}
 4s,\quad m+15\ell\equiv 0(s);\\
 5s,\quad m+15\ell\equiv 1(s);\\
 0,\quad\text{otherwise.}
 \end{cases}$$

$(6)$ If $r\in\{5,23\}$, then $$\dim_K\Hom_\Lambda(Q_{deg},
R)=\begin{cases}
 11s,\quad m+15\ell\equiv 0(s);\\
 0,\quad\text{otherwise.}
 \end{cases}$$

$(7)$ If $r=6$, then $$\dim_K\Hom_\Lambda(Q_{deg},
R)=\begin{cases}
 6s,\quad m+15\ell\equiv 0(s);\\
 7s,\quad m+15\ell\equiv 1(s);\\
 0,\quad\text{otherwise.}
 \end{cases}$$

$(8)$ If $r\in\{7,21\}$, then $$\dim_K\Hom_\Lambda(Q_{deg},
R)=\begin{cases}
 14s,\quad m+15\ell\equiv 0(s);\\
 0,\quad\text{otherwise.}
 \end{cases}$$
 
$(9)$ If $r=8$, then $$\dim_K\Hom_\Lambda(Q_{deg},
R)=\begin{cases}
 4s,\quad m+15\ell\equiv 0(s);\\
 8s,\quad m+15\ell\equiv 1(s);\\
 0,\quad\text{otherwise.}
 \end{cases}$$
 
$(10)$ If $r\in\{9,19\}$, then $$\dim_K\Hom_\Lambda(Q_{deg},
R)=\begin{cases}
 16s,\quad m+15\ell\equiv 0(s);\\
 0,\quad\text{otherwise.}
 \end{cases}$$
 
$(11)$ If $r=10$, then $$\dim_K\Hom_\Lambda(Q_{deg},
R)=\begin{cases}
 11s,\quad m+15\ell\equiv 0(s);\\
 12s,\quad m+15\ell\equiv 1(s);\\
 0,\quad\text{otherwise.}
 \end{cases}$$
 
$(12)$ If $r\in\{11,13,15,17\}$, then $$\dim_K\Hom_\Lambda(Q_{deg},
R)=\begin{cases}
 18s,\quad m+15\ell\equiv 0(s);\\
 0,\quad\text{otherwise.}
 \end{cases}$$
 
$(13)$ If $r=12$, then $$\dim_K\Hom_\Lambda(Q_{deg},
R)=\begin{cases}
 10s,\quad m+15\ell\equiv 0(s);\\
 7s,\quad m+15\ell\equiv 1(s);\\
 0,\quad\text{otherwise.}
 \end{cases}$$
 
$(14)$ If $r=16$, then $$\dim_K\Hom_\Lambda(Q_{deg},
R)=\begin{cases}
 7s,\quad m+15\ell\equiv 0(s);\\
 10s,\quad m+15\ell\equiv 1(s);\\
 0,\quad\text{otherwise.}
 \end{cases}$$
 
$(15)$ If $r=18$, then $$\dim_K\Hom_\Lambda(Q_{deg},
R)=\begin{cases}
 12s,\quad m+15\ell\equiv 0(s);\\
 11s,\quad m+15\ell\equiv 1(s);\\
 0,\quad\text{otherwise.}
 \end{cases}$$
 
$(16)$ If $r=20$, then $$\dim_K\Hom_\Lambda(Q_{deg},
R)=\begin{cases}
 8s,\quad m+15\ell\equiv 0(s);\\
 4s,\quad m+15\ell\equiv 1(s);\\
 0,\quad\text{otherwise.}
 \end{cases}$$
 
$(17)$ If $r=22$, then $$\dim_K\Hom_\Lambda(Q_{deg},
R)=\begin{cases}
 7s,\quad m+15\ell\equiv 0(s);\\
 6s,\quad m+15\ell\equiv 1(s);\\
 0,\quad\text{otherwise.}
 \end{cases}$$
 
$(18)$ If $r=24$, then $$\dim_K\Hom_\Lambda(Q_{deg},
R)=\begin{cases}
 5s,\quad m+15\ell\equiv 0(s);\\
 4s,\quad m+15\ell\equiv 1(s);\\
 0,\quad\text{otherwise.}
 \end{cases}$$
 
$(19)$ If $r=26$, then $$\dim_K\Hom_\Lambda(Q_{deg},
R)=\begin{cases}
 s,\quad m+15\ell\equiv 0(s);\\
 5s,\quad m+15\ell\equiv 1(s);\\
 0,\quad\text{otherwise.}
 \end{cases}$$
 
\end{pr}

\begin{proof}
The dimension $\dim_K\Hom_\Lambda(P_{i,j}, R)$ is equal to the number of linear independent nonzero
paths of the quiver $\mathcal Q^\prime_s$, going from the $j$th vertex to the $i$th vertex. 
Thus the proof is to consider the cases $r=0$, $r=1$ etc.
\end{proof}

\begin{pr}[Dimensions of homomorphism groups, $s=1$]

Let $R=R^\prime_1$ be an algebra of type $E_8$. Next, let $deg\in\N\cup\{0\}$,
and $\ell$ and $r$ the integral part and residue of $deg$ modulo $29$. Then

$(1)$ If $r\in\{0,9,14,19,28\}$, then $\dim_K\Hom_\Lambda(Q_{deg}, R)=16$.

$(2)$ If $r\in\{1,4,24,27\}$, then $\dim_K\Hom_\Lambda(Q_{deg}, R)=9$.

$(3)$ If $r\in\{2,26\}$, then $\dim_K\Hom_\Lambda(Q_{deg}, R)=6$.

$(4)$ If $r\in\{3,25\}$, then $\dim_K\Hom_\Lambda(Q_{deg}, R)=10$.

$(5)$ If $r\in\{5,23\}$, then $\dim_K\Hom_\Lambda(Q_{deg}, R)=11$.

$(6)$ If $r\in\{6,22\}$, then $\dim_K\Hom_\Lambda(Q_{deg}, R)=13$.

$(7)$ If $r\in\{7,21\}$, then $\dim_K\Hom_\Lambda(Q_{deg}, R)=14$.

$(8)$ If $r\in\{8,20\}$, then $\dim_K\Hom_\Lambda(Q_{deg}, R)=12$.

$(9)$ If $r\in\{10,18\}$, then $\dim_K\Hom_\Lambda(Q_{deg}, R)=23$.

$(10)$ If $r\in\{11,13,15,17\}$, then $\dim_K\Hom_\Lambda(Q_{deg}, R)=18$.

$(11)$ If $r\in\{12,16\}$, then $\dim_K\Hom_\Lambda(Q_{deg}, R)=17$.

\end{pr}

\begin{proof}
The proof is basically the same as the proof of Proposition \ref{dim_hom8}.
\end{proof}

\begin{pr}[Dimensions of coboundaries groups]\label{dim_im8}
Let $R=R^\prime_s$ be an algebra of type $E_8$, and let
\begin{equation}\tag{$\times\times$}\label{ind_resolv8} 0\longrightarrow
\Hom_\Lambda(Q_0, R)\stackrel{\delta^0}\longrightarrow \Hom_\Lambda(Q_1,
R)\stackrel{\delta^1}\longrightarrow \Hom_\Lambda(Q_2, R)\stackrel{\delta^2}\longrightarrow\dots
\end{equation} be a complex obtained from the minimal projective
resolution \eqref{resolv8} of algebra $R$ by applying the functor
$\Hom_\Lambda(-,R)$.

Consider the coboundaries groups $\Im\delta^{deg}$ of the complex
\eqref{ind_resolv8}. Let $\ell$ and $r$ be the integral part and residue of
$deg$ modulo $29$, and $m$ the integral part of $r$ modulo $2$. Then

$(1)$ If $r\in\{0,7,20,27\}$, then $$\dim_K\Im\delta^{deg}=\begin{cases}
 8s-1,\quad m+15\ell\equiv 0(s),\text{ }\ell+m\div 2\text{ or }\myChar=2;\\
 8s,\quad m+15\ell\equiv 0(s),\text{ }\ell+m\ndiv 2,\text{ }\myChar\ne 2;\\
 0,\quad\text{otherwise.}
 \end{cases}$$

$(2)$ If $r\in\{1,26\}$, then $$\dim_K\Im\delta^{deg}=\begin{cases}
 s,\quad m+15\ell\equiv 0(s);\\
 0,\quad\text{otherwise.}
 \end{cases}$$

$(3)$ If $r\in\{2,25\}$, then $$\dim_K\Im\delta^{deg}=\begin{cases}
 5s,\quad m+15\ell\equiv 0(s);\\
 0,\quad\text{otherwise.}
 \end{cases}$$

$(4)$ If $r\in\{3,24\}$, then $$\dim_K\Im\delta^{deg}=\begin{cases}
 5s-1,\quad m+15\ell\equiv 0(s),\text{ }\ell+m\div 2\text{ or }\myChar=2;\\
 5s,\quad m+15\ell\equiv 0(s),\text{ }\ell+m\ndiv 2,\text{ }\myChar\ne 2;\\
 0,\quad\text{otherwise.}
 \end{cases}$$

$(5)$ If $r\in\{4,8,19,23\}$, then $$\dim_K\Im\delta^{deg}=\begin{cases}
 4s,\quad m+15\ell\equiv 0(s);\\
 0,\quad\text{otherwise.}
 \end{cases}$$

$(6)$ If $r\in\{5,22\}$, then $$\dim_K\Im\delta^{deg}=\begin{cases}
 7s-1,\quad m+15\ell\equiv 0(s),\text{ }\ell+m\div 2\text{ and }\myChar=3;\\
 7s,\quad m+15\ell\equiv 0(s),\text{ }\ell+m\ndiv 2\text{ or }\myChar\ne 3;\\
 0,\quad\text{otherwise.}
 \end{cases}$$

$(7)$ If $r\in\{6,21\}$, then $$\dim_K\Im\delta^{deg}=\begin{cases}
 6s,\quad m+15\ell\equiv 0(s);\\
 0,\quad\text{otherwise.}
 \end{cases}$$

$(8)$ If $r\in\{9,18\}$, then $$\dim_K\Im\delta^{deg}=\begin{cases}
 12s-1,\quad m+15\ell\equiv 0(s),\text{ }\ell+m\div 2\text{ and }\myChar=5;\\
 12s,\quad m+15\ell\equiv 0(s),\text{ }\ell+m\ndiv 2\text{ or }\myChar\ne 5;\\
 0,\quad\text{otherwise.}
 \end{cases}$$
 
$(9)$ If $r\in\{10,17\}$, then $$\dim_K\Im\delta^{deg}=\begin{cases}
 11s-1,\quad m+15\ell\equiv 0(s),\text{ }\ell+m\div 2\text{ and }\myChar=3;\\
 11s,\quad m+15\ell\equiv 0(s),\text{ }\ell+m\ndiv 2\text{ or }\myChar\ne 3;\\
 0,\quad\text{otherwise.}
 \end{cases}$$
 
$(10)$ If $r\in\{11,16\}$, then $$\dim_K\Im\delta^{deg}=\begin{cases}
 7s,\quad m+15\ell\equiv 0(s);\\
 0,\quad\text{otherwise.}
 \end{cases}$$
 
$(11)$ If $r\in\{12,15\}$, then $$\dim_K\Im\delta^{deg}=\begin{cases}
 10s-1,\quad m+15\ell\equiv 0(s),\text{ }\ell+m\div 2\text{ or }\myChar=2;\\
 10s,\quad m+15\ell\equiv 0(s),\text{ }\ell+m\ndiv 2,\text{ }\myChar\ne 2;\\
 0,\quad\text{otherwise.}
 \end{cases}$$
 
$(12)$ If $r\in\{13,14,28\}$, then $$\dim_K\Im\delta^{deg}=\begin{cases}
 8s,\quad m+15\ell\equiv 0(s);\\
 0,\quad\text{otherwise.}
 \end{cases}$$
 
\end{pr}

\begin{proof}
The proof is technical and consists in constructing the image matrices 
from the description of the differential matrices, and the subsequent calculations of the ranks of image matrices.
\end{proof}

\begin{thm}[Additive structure, $s>1$]
Let $s>1$ and $R=R^\prime_s$ is algebra of the type $E_8$. Next, let $deg\in\N\cup\{0\}$,
$\ell$ and $r$ the integral part and residue of $deg$ modulo $29$,
and $m$ the integral part of $r$ modulo $2$. Then
$\dim_K\HH^{deg}(R)=1$ if one of the following conditions is satisfied$:$

$(1)$ $r\in \{0,1,3,7,12,13,15,20,21,24,25,27\}$, $m+15\ell\equiv 0(s),\text{ }\ell+m\div 2\text{ or }\myChar=2$;

$(2)$ $r\in \{0,4,8,16,28\}$, $m+15\ell\equiv 1(s),\text{ }\ell+m\ndiv 2\text{ or }\myChar=2$;

$(3)$ $r\in \{5,10,11,17,22,23\}$, $m+15\ell\equiv 0(s),\text{ }\ell+m\div 2,\text{ }\myChar=3$;

$(4)$ $r\in \{6,18\}$, $m+15\ell\equiv 1(s),\text{ }\ell+m\ndiv 2,\text{ }\myChar=3$;

$(5)$ $r\in \{9,18,19\}$, $m+15\ell\equiv 0(s),\text{ }\ell+m\div 2,\text{ }\myChar=5$;

$(6)$ $r=10$, $m+15\ell\equiv 1(s),\text{ }\ell+m\ndiv 2,\text{ }\myChar=3$.

In the other cases, $\dim_K\HH^{deg}(R)=0$.
\end{thm}

\begin{proof}
As $\dim_K\HH^{deg}(R)=\dim_K\Ker\delta^{deg}-\dim_K\Im\delta^{deg-1}$, and\\
$\dim_K\Ker\delta^{deg}=\dim_K\Hom_\Lambda(Q_{deg},R)-\dim_K\Im\delta^{deg}$,
the assertions of the theorem easily follow from Propositions
\ref{dim_hom8} -- \ref{dim_im8}.
\end{proof}

\begin{thm}[Additive structure, $s=1$]
Let $R=R^\prime_1$ is algebra of the type $E_8$. Next, let $deg\in\N\cup\{0\}$,
$\ell$ and $r$ the integral part and residue of $deg$ modulo $29$,
and $m$ the integral part of $r$ modulo $2$. Then\\
$($a$)$ $\dim_K\HH^{deg}(R)=8$, if $deg=0$.\\
$($b$)$ $\dim_K\HH^{deg}(R)=1$, if one of the following conditions is satisfied$:$

$(1)$ $r\in \{0,1,3,7,12,13,15,20,21,24,25,27\}$, $deg>0$, $\ell+m\div 2\text{ or }\myChar=2$;

$(2)$ $r\in \{0,4,8,16,28\}$, $\ell+m\ndiv 2\text{ or }\myChar=2$;

$(3)$ $r\in \{5,10,11,17,22,23\}$, $\ell+m\div 2,\text{ }\myChar=3$;

$(4)$ $r\in \{6,18\}$, $\ell+m\ndiv 2,\text{ }\myChar=3$;

$(5)$ $r\in \{9,18,19\}$, $\ell+m\div 2,\text{ }\myChar=5$;

$(6)$ $r=10$, $\ell+m\ndiv 2,\text{ }\myChar=3$.

$($c$)$ In the other cases, $\dim_K\HH^{deg}(R)=0$.

\end{thm}

\section{Tree class $E_8$: Generators of $\HH^*(R)$}\label{gen8}

For $s>1$, we introduce a set of generators $Z^{(1)}_t$, $Z^{(2)}_t$, \dots $Z^{(28)}_t$, such that
$\deg Z_t^{(i)}=t$, $0\le t < 29\ord\rho$ and $t$ satisfies the conditions (i) from the list given in
Sec. \ref{res8}. For $s=1$, we introduce a set of
generators $Z^{(1)}_t$, $Z^{(2)}_t$, \dots $Z^{(36)}_t$, 
such that $\deg Z_t^{(i)}=t$, $0\le t < 29\ord\rho$ and $t$ 
satisfies the conditions (i) from the list given in Sec. \ref{res8} for
$i\le 28$ and $t=0$ if $i>28$. For the generator $Q_t\rightarrow R$ 
we describe the map $Q_t\rightarrow Q_0$ as matrix. The corresponding generator is the composition
of this map with multiplication map $Q_t\rightarrow Q_0\stackrel\varepsilon\longrightarrow R$.

\begin{obozn}
For $j$th column $j_2$ is the integral part of $j$ modulo $s$.
\end{obozn}

(1) $Z^{(1)}_t$ is a $(8s\times 8s)$ matrix, whose entries $z_{ij}$ have the following form:
$$z_{ij}=
\begin{cases}
e_{8j+j_2}\otimes e_{8j+j_2},\quad i=j;\\
0,\quad\text{otherwise.}\end{cases}$$

(2) $Z^{(2)}_t$ is a $(8s\times 9s)$ matrix with two nonzero entries
$$z_{0,0}=w_{0\ra 1}\otimes e_{0}\text{ and } z_{0,s}=w_{0\ra 5}\otimes e_{0}.$$

(3) $Z^{(3)}_t$ is a $(8s\times 10s)$ matrix, whose entries $z_{ij}$ have the following form:

If $0\le j<s$, then $$z_{ij}=
\begin{cases}w_{8j\ra 8j+2}\otimes e_{8j},\quad i=j;\\
0,\quad\text{otherwise.}\end{cases}$$

If $s\le j<2s$, then $$z_{ij}=
\begin{cases}-w_{8j\ra 8j+6}\otimes e_{8j},\quad i=j-s;\\
0,\quad\text{otherwise.}\end{cases}$$

If $2s\le j<5s$, then $z_{ij}=0$.

If $5s\le j<6s$, then $$z_{ij}=
\begin{cases}w_{8j+4\ra 8(j+1)}\otimes e_{8j+4},\quad i=j-s;\\
0,\quad\text{otherwise.}\end{cases}$$

If $6s\le j<8s$, then $z_{ij}=0$.

If $8s\le j<9s$, then $$z_{ij}=
\begin{cases}w_{8j+7\ra 8(j+1)+1}\otimes e_{8j+7},\quad i=j-s;\\
0,\quad\text{otherwise.}\end{cases}$$

If $9s\le j<10s$, then $z_{ij}=0$.

(4) $Z^{(4)}_t$ is a $(8s\times 11s)$ matrix with unique nonzero entry
$$z_{0,0}=w_{0\ra 7}\otimes e_0.$$
  
(5) $Z^{(5)}_t$ is a $(8s\times 11s)$ matrix, whose entries $z_{ij}$ have the following form:

If $0\le j<3s$, then $$z_{ij}=
\begin{cases}(1-2f(j_2,2))w_{8j\ra 8j+5-2j_2}\otimes e_{8j},\quad i=(j)_s;\\
0,\quad\text{otherwise.}\end{cases}$$

If $3s\le j<5s$, then $z_{ij}=0$.

If $5s\le j<8s$, then $$z_{ij}=
\begin{cases}(-1)^{j_2+1}w_{8j+j_2-2\ra 8(j+1)-f(j_2,6)}\otimes e_{8j+j_2-2},\quad i=j-2s;\\
0,\quad\text{otherwise.}\end{cases}$$

If $8s\le j<11s$, then $z_{ij}=0$.
 
(6) $Z^{(6)}_t$ is a $(8s\times 13s)$ matrix with unique nonzero entry
$$z_{0,0}=w_{0\ra 7}\otimes e_0.$$

(7) $Z^{(7)}_t$ is a $(8s\times 14s)$ matrix, whose entries $z_{ij}$ have the following form:

If $0\le j<s$, then $z_{ij}=0$.

If $s\le j<2s$, then $$z_{ij}=
\begin{cases}w_{8j\ra 8j+4}\otimes e_{8j},\quad i=j-s;\\
0,\quad\text{otherwise.}\end{cases}$$

If $2s\le j<3s$, then $z_{ij}=0$.

If $3s\le j<4s$, then $$z_{ij}=
\begin{cases}w_{8j\ra 8j+5}\otimes e_{8j},\quad i=j-3s;\\
0,\quad\text{otherwise.}\end{cases}$$

If $4s\le j<9s$, then $z_{ij}=0$.

If $9s\le j<10s$, then $$z_{ij}=
\begin{cases}w_{8j+6\ra 8j+7}\otimes e_{8j+6},\quad i=j-3s;\\
0,\quad\text{otherwise.}\end{cases}$$

If $10s\le j<11s$, then $z_{ij}=0$.

If $11s\le j<12s$, then $$z_{ij}=
\begin{cases}w_{8j+7\ra 8(j+1)+5}\otimes e_{8j+7},\quad i=j-4s;\\
0,\quad\text{otherwise.}\end{cases}$$

If $12s\le j<14s$, then $z_{ij}=0$.

(8) $Z^{(8)}_t$ is a $(8s\times 16s)$ matrix with unique nonzero entry
$$z_{0,0}=-w_{0\ra 7}\otimes e_0.$$

(9) $Z^{(9)}_t$ is a $(8s\times 16s)$ matrix, whose entries $z_{ij}$ have the following form:

If $0\le j<3s$, then $$z_{ij}=
\begin{cases}(-1)^{j_2+1}w_{8j\ra 8j+3j_2+5f(j_2,0)}\otimes e_{8j},\quad i=(j)_s;\\
0,\quad\text{otherwise.}\end{cases}$$

If $3s\le j<5s$, then $z_{ij}=0$.

If $5s\le j<8s$, then $$z_{ij}=
\begin{cases}(1+f(j_2,6))w_{8j+j_2-3\ra 8j+7+f(j_2,6)}\otimes e_{8j+j_2-3},\quad i=j-3s;\\
0,\quad\text{otherwise.}\end{cases}$$

If $8s\le j<9s$, then $z_{ij}=0$.

If $9s\le j<10s$, then $$z_{ij}=
\begin{cases}w_{8j+5\ra 8(j+1)}\otimes e_{8j+5},\quad i=j-4s;\\
0,\quad\text{otherwise.}\end{cases}$$

If $10s\le j<12s$, then $z_{ij}=0$.

If $12s\le j<14s$, then $$z_{ij}=
\begin{cases}(1+2f(j_2,13))w_{8j+7\ra 8(j+1)+2(j_2-11)}\otimes e_{8j+7},\quad i=7s+(j)_s;\\
0,\quad\text{otherwise.}\end{cases}$$

If $14s\le j<16s$, then $z_{ij}=0$.

(10) $Z^{(10)}_t$ is a $(8s\times 19s)$ matrix, whose entries $z_{ij}$ have the following form:

If $0\le j<s$, then $z_{ij}=0$.

If $s\le j<3s$, then $$z_{ij}=
\begin{cases}(-1)^{j_2+1}e_{8j}\otimes e_{8j},\quad i=(j)_s;\\
0,\quad\text{otherwise.}\end{cases}$$

If $3s\le j<4s$, then $z_{ij}=0$.

If $4s\le j<7s$, then $$z_{ij}=
\begin{cases}(-1+2f(j_2,4))e_{8j+j_2-3}\otimes e_{8j+j_2-3},\quad i=j-3s;\\
0,\quad\text{otherwise.}\end{cases}$$

If $7s\le j<8s$, then $z_{ij}=0$.

If $8s\le j<9s$, then $$z_{ij}=
\begin{cases}-e_{8j+4}\otimes e_{8j+4},\quad i=j-4s;\\
0,\quad\text{otherwise.}\end{cases}$$

If $9s\le j<10s$, then $z_{ij}=0$.

If $10s\le j<11s$, then $$z_{ij}=
\begin{cases}e_{8j+5}\otimes e_{8j+5},\quad i=j-5s;\\
0,\quad\text{otherwise.}\end{cases}$$

If $11s\le j<14s$, then $z_{ij}=0$.

If $14s\le j<15s$, then $$z_{ij}=
\begin{cases}e_{8j+6}\otimes e_{8j+6},\quad i=j-8s;\\
0,\quad\text{otherwise.}\end{cases}$$

If $15s\le j<16s$, then $z_{ij}=0$.

If $16s\le j<19s$, then $$z_{ij}=
\begin{cases}(-1)^{j_2}w_{8j+7\ra 8j+7+f(j_2,18)}\otimes e_{8j+7},\quad i=7s+(j)_s;\\
0,\quad\text{otherwise.}\end{cases}$$

(11) $Z^{(11)}_t$ is a $(8s\times 19s)$ matrix with two nonzero entries
$$z_{0,s}=w_{0\ra 8}\otimes e_{0}\text{ and } z_{0,2s}=w_{0\ra 8}\otimes e_{0}.$$

(12) $Z^{(12)}_t$ is a $(8s\times 18s)$ matrix with the following nonzero entries
$$z_{(-5)_s,(-5)_s}=w_{8j\ra 8j+2}\otimes e_{8j};$$
$$z_{(-4)_s,(-4)_s}=-w_{8j\ra 8j+2}\otimes e_{8j};$$
$$z_{(-3)_s,(-3)_s}=w_{8j\ra 8j+2}\otimes e_{8j};$$
$$z_{(-2)_s,(-2)_s}=-w_{8j\ra 8j+2}\otimes e_{8j};$$
$$z_{(-6)_s,s+(-6)_s}=-w_{8j\ra 8j+6}\otimes e_{8j};$$
$$z_{(-2)_s,s+(-2)_s}=w_{8j\ra 8j+6}\otimes e_{8j};$$
$$z_{2s+(-5)_s,6s+(-5)_s}=-w_{8j+2\ra 8(j+1)}\otimes e_{8j+2};$$
$$z_{2s+(-4)_s,6s+(-4)_s}=w_{8j+2\ra 8(j+1)}\otimes e_{8j+2};$$
$$z_{2s+(-3)_s,6s+(-3)_s}=-w_{8j+2\ra 8(j+1)}\otimes e_{8j+2};$$
$$z_{3s+(-6)_s,7s+(-6)_s}=-w_{8j+3\ra 8j+7}\otimes e_{8j+3};$$
$$z_{3s+(-5)_s,7s+(-5)_s}=w_{8j+3\ra 8j+7}\otimes e_{8j+3};$$
$$z_{3s+(-4)_s,7s+(-4)_s}=-w_{8j+3\ra 8j+7}\otimes e_{8j+3};$$
$$z_{3s+(-3)_s,7s+(-3)_s}=w_{8j+3\ra 8j+7}\otimes e_{8j+3};$$
$$z_{4s+(-2)_s,8s+(-2)_s}=-w_{8j+4\ra 8j+7}\otimes e_{8j+4};$$
$$z_{5s+(-6)_s,10s+(-6)_s}=w_{8j+5\ra 8j+7}\otimes e_{8j+5};$$
$$z_{5s+(-5)_s,10s+(-5)_s}=-w_{8j+5\ra 8j+7}\otimes e_{8j+5};$$
$$z_{5s+(-4)_s,10s+(-4)_s}=w_{8j+5\ra 8j+7}\otimes e_{8j+5};$$
$$z_{5s+(-3)_s,10s+(-3)_s}=-w_{8j+5\ra 8j+7}\otimes e_{8j+5};$$
$$z_{5s+(-2)_s,10s+(-2)_s}=-w_{8j+5\ra 8j+7}\otimes e_{8j+5}.$$

(13) $Z^{(13)}_t$ is a $(8s\times 19s)$ matrix, whose entries $z_{ij}$ have the following form:

If $0\le j<s$, then $z_{ij}=0$.

If $s\le j<2s$, then $$z_{ij}=
\begin{cases}e_{8j}\otimes e_{8j},\quad i=j-s;\\
0,\quad\text{otherwise.}\end{cases}$$

If $2s\le j<3s$, then $z_{ij}=0$.

If $3s\le j<5s$, then $$z_{ij}=
\begin{cases}-w_{8j+j_2-2\ra 8j+j_2-1}\otimes e_{8j+j_2-2},\quad i=j-2s;\\
0,\quad\text{otherwise.}\end{cases}$$

If $5s\le j<6s$, then $z_{ij}=0$.

If $6s\le j<7s$, then $$z_{ij}=
\begin{cases}-w_{8j+3\ra 8j+4}\otimes e_{8j+3},\quad i=j-3s;\\
0,\quad\text{otherwise.}\end{cases}$$

If $7s\le j<11s$, then $z_{ij}=0$.

If $11s\le j<12s$, then $$z_{ij}=
\begin{cases}w_{8j+5\ra 8j+6}\otimes e_{8j+5},\quad i=j-6s;\\
0,\quad\text{otherwise.}\end{cases}$$

If $12s\le j<14s$, then $$z_{ij}=
\begin{cases}(-1)^{j_2+1}e_{8j+j_2-7}\otimes e_{8j+j_2-7},\quad i=j-7s;\\
0,\quad\text{otherwise.}\end{cases}$$

If $14s\le j<16s$, then $z_{ij}=0$.

If $16s\le j<17s$, then $$z_{ij}=
\begin{cases}e_{8j+7}\otimes e_{8j+7},\quad i=j-9s;\\
0,\quad\text{otherwise.}\end{cases}$$

If $17s\le j<18s$, then $z_{ij}=0$.

If $18s\le j<19s$, then $$z_{ij}=
\begin{cases}w_{8j+7\ra 8(j+1)}\otimes e_{8j+7},\quad i=j-11s;\\
0,\quad\text{otherwise.}\end{cases}$$

(14) $Z^{(14)}_t$ is a $(8s\times 18s)$ matrix with the following nonzero entries
$$z_{(-3)_s,s+(-3)_s}=w_{8j\ra 8j+3}\otimes e_{8j};$$
$$z_{s+(-3)_s,4s+(-3)_s}=-w_{8j+1\ra 8j+8}\otimes e_{8j+1};$$
$$z_{5s+(-3)_s,9s+(-3)_s}=w_{8j+5\ra 8j+7}\otimes e_{8j+5};$$
$$z_{7s+(-4)_s,14s+(-4)_s}=w_{8j+7\ra 8(j+1)+6}\otimes e_{8j+7}.$$

(15) $Z^{(15)}_t$ is a $(8s\times 18s)$ matrix, whose entries $z_{ij}$ have the following form:

If $0\le j<s$, then $$z_{ij}=
\begin{cases}w_{8j\ra 8j+3}\otimes e_{8j},\quad i=j;\\
0,\quad\text{otherwise.}\end{cases}$$

If $s\le j<8s$, then $z_{ij}=0$.

If $8s\le j<9s$, then $$z_{ij}=
\begin{cases}-w_{8j+3\ra 8(j+1)}\otimes e_{8j+3},\quad i=j-5s;\\
0,\quad\text{otherwise.}\end{cases}$$

If $9s\le j<10s$, then $z_{ij}=0$.

If $10s\le j<11s$, then $$z_{ij}=
\begin{cases}-w_{8j+5\ra 8j+7}\otimes e_{8j+5},\quad i=j-5s;\\
0,\quad\text{otherwise.}\end{cases}$$

If $11s\le j<14s$, then $z_{ij}=0$.

If $14s\le j<15s$, then $$z_{ij}=
\begin{cases}w_{8j+7\ra 8(j+1)+2}\otimes e_{8j+7},\quad i=j-7s;\\
0,\quad\text{otherwise.}\end{cases}$$

If $15s\le j<18s$, then $z_{ij}=0$.

(16) $Z^{(16)}_t$ is a $(8s\times 19s)$ matrix with unique nonzero entry
$$z_{0,0}=w_{0\ra 8}\otimes e_0.$$

(17) $Z^{(17)}_t$ is a $(8s\times 18s)$ matrix, whose entries $z_{ij}$ have the following form:

If $0\le j<s$, then $z_{ij}=0$.

If $s\le j<2s$, then $$z_{ij}=
\begin{cases}-w_{8j\ra 8j+4}\otimes e_{8j},\quad i=j-s;\\
0,\quad\text{otherwise.}\end{cases}$$

If $2s\le j<3s$, then $z_{ij}=0$.

If $3s\le j<4s$, then $$z_{ij}=
\begin{cases}w_{8j\ra 8j+5}\otimes e_{8j},\quad i=j-3s;\\
0,\quad\text{otherwise.}\end{cases}$$

If $4s\le j<7s$, then $z_{ij}=0$.

If $7s\le j<8s$, then $$z_{ij}=
\begin{cases}-w_{8j+3\ra 8j+7}\otimes e_{8j+3},\quad i=j-4s;\\
0,\quad\text{otherwise.}\end{cases}$$

If $8s\le j<9s$, then $z_{ij}=0$.

If $9s\le j<10s$, then $$z_{ij}=
\begin{cases}-w_{8j+5\ra 8(j+1)}\otimes e_{8j+5},\quad i=j-4s;\\
0,\quad\text{otherwise.}\end{cases}$$

If $10s\le j<13s$, then $z_{ij}=0$.

If $13s\le j<14s$, then $$z_{ij}=
\begin{cases}w_{8j+7\ra 8(j+1)+3}\otimes e_{8j+7},\quad i=j-6s;\\
0,\quad\text{otherwise.}\end{cases}$$

If $14s\le j<18s$, then $z_{ij}=0$.

(18) $Z^{(18)}_t$ is a $(8s\times 19s)$ matrix, whose entries $z_{ij}$ have the following form:

If $0\le j<s$, then $z_{ij}=0$.

If $s\le j<2s$, then $$z_{ij}=
\begin{cases}e_{8j}\otimes e_{8j},\quad i=j-s;\\
0,\quad\text{otherwise.}\end{cases}$$

If $2s\le j<3s$, then $$z_{ij}=
\begin{cases}2e_{8j}\otimes e_{8j},\quad i=j-2s;\\
0,\quad\text{otherwise.}\end{cases}$$

If $3s\le j<4s$, then $z_{ij}=0$.

If $4s\le j<6s$, then $$z_{ij}=
\begin{cases}2(-1)^{j_2+1}e_{8j+j_2-3}\otimes e_{8j+j_2-3},\quad i=j-3s;\\
0,\quad\text{otherwise.}\end{cases}$$

If $6s\le j<7s$, then $z_{ij}=0$.

If $7s\le j<9s$, then $$z_{ij}=
\begin{cases}2e_{8j+j_2-4}\otimes e_{8j+j_2-4},\quad i=j-4s;\\
0,\quad\text{otherwise.}\end{cases}$$

If $9s\le j<11s$, then $z_{ij}=0$.

If $11s\le j<12s$, then $$z_{ij}=
\begin{cases}2w_{8j+5\ra 8j+6}\otimes e_{8j+5},\quad i=j-6s;\\
0,\quad\text{otherwise.}\end{cases}$$

If $12s\le j<14s$, then $$z_{ij}=
\begin{cases}(-1)^{j_2+1}e_{8j+j_2-7}\otimes e_{8j+j_2-7},\quad i=j-7s;\\
0,\quad\text{otherwise.}\end{cases}$$

If $14s\le j<16s$, then $z_{ij}=0$.

If $16s\le j<17s$, then $$z_{ij}=
\begin{cases}e_{8j+7}\otimes e_{8j+7},\quad i=j-9s;\\
0,\quad\text{otherwise.}\end{cases}$$

If $17s\le j<18s$, then $z_{ij}=0$.

If $18s\le j<19s$, then $$z_{ij}=
\begin{cases}2e_{8j+7}\otimes e_{8j+7},\quad i=j-11s;\\
0,\quad\text{otherwise.}\end{cases}$$

(19) $Z^{(19)}_t$ is a $(8s\times 19s)$ matrix with unique nonzero entry
$$z_{0,0}=w_{0\ra 8}\otimes e_0.$$

(20) $Z^{(20)}_t$ is a $(8s\times 16s)$ matrix with the following nonzero entries
$$z_{(-3)_s,s+(-3)_s}=-w_{8j\ra 8j+6}\otimes e_{8j};$$
$$z_{(-2)_s,s+(-2)_s}=-w_{8j\ra 8j+6}\otimes e_{8j};$$
$$z_{s+(-4)_s,4s+(-4)_s}=-w_{8j+1\ra 8(j+1)}\otimes e_{8j+1};$$
$$z_{2s+(-3)_s,5s+(-3)_s}=w_{8j+2\ra 8j+7}\otimes e_{8j+2};$$
$$z_{2s+(-2)_s,5s+(-2)_s}=w_{8j+2\ra 8j+7}\otimes e_{8j+2};$$
$$z_{5s+(-4)_s,9s+(-4)_s}=w_{8j+5\ra 8(j+1)}\otimes e_{8j+5}.$$

(21) $Z^{(21)}_t$ is a $(8s\times 16s)$ matrix, whose entries $z_{ij}$ have the following form:

If $0\le j<s$, then $z_{ij}=0$.

If $s\le j<3s$, then $$z_{ij}=
\begin{cases}(-1)^{j_2+1} w_{8j+j_2-1\ra 8j+2(j_2-1)}\otimes e_{8j+j_2-1},\quad i=j-s;\\
0,\quad\text{otherwise.}\end{cases}$$

If $3s\le j<4s$, then $z_{ij}=0$.

If $4s\le j<6s$, then $$z_{ij}=
\begin{cases}-w_{8j+j_2-2\ra 8j+j_2-1}\otimes e_{8j+j_2-2},\quad i=j-2s;\\
0,\quad\text{otherwise.}\end{cases}$$

If $6s\le j<8s$, then $z_{ij}=0$.

If $8s\le j<9s$, then $$z_{ij}=
\begin{cases}-w_{8j+5\ra 8j+6}\otimes e_{8j+5},\quad i=j-3s;\\
0,\quad\text{otherwise.}\end{cases}$$

If $9s\le j<13s$, then $z_{ij}=0$.

If $13s\le j<15s$, then $$z_{ij}=
\begin{cases}(-1)^{j_2+1} w_{8j+7\ra 8j+j_2-6}\otimes e_{8j+7},\quad i=7s+(j)_s;\\
0,\quad\text{otherwise.}\end{cases}$$

If $15s\le j<16s$, then $z_{ij}=0$.

(22) $Z^{(22)}_t$ is a $(8s\times 14s)$ matrix with the following nonzero entries
$$z_{3s+(-7)_s,5s+(-7)_s}=-w_{8j+3\ra 8j+7}\otimes e_{8j+3};$$
$$z_{3s+(-6)_s,5s+(-6)_s}=w_{8j+3\ra 8j+7}\otimes e_{8j+3};$$
$$z_{3s+(-5)_s,5s+(-5)_s}=-w_{8j+3\ra 8j+7}\otimes e_{8j+3};$$
$$z_{5s+(-7)_s,7s+(-7)_s}=-w_{8j+5\ra 8j+7}\otimes e_{8j+5};$$
$$z_{5s+(-6)_s,7s+(-6)_s}=w_{8j+5\ra 8j+7}\otimes e_{8j+5};$$
$$z_{5s+(-5)_s,7s+(-5)_s}=-w_{8j+5\ra 8j+7}\otimes e_{8j+5}.$$

(23) $Z^{(23)}_t$ is a $(8s\times 13s)$ matrix, whose entries $z_{ij}$ have the following form:

If $0\le j<s$, then $$z_{ij}=
\begin{cases}e_{8j}\otimes e_{8j},\quad i=j;\\
0,\quad\text{otherwise.}\end{cases}$$

If $s\le j<2s$, then $z_{ij}=0$.

If $2s\le j<4s$, then $$z_{ij}=
\begin{cases}-w_{8j+j_2-1\ra 8j+j_2+1}\otimes e_{8j+j_2-1},\quad i=j-s;\\
0,\quad\text{otherwise.}\end{cases}$$

If $4s\le j<8s$, then $z_{ij}=0$.

If $8s\le j<9s$, then $$z_{ij}=
\begin{cases}-e_{8j+5}\otimes e_{8j+5},\quad i=j-3s;\\
0,\quad\text{otherwise.}\end{cases}$$

If $9s\le j<10s$, then $z_{ij}=0$.

If $10s\le j<12s$, then $$z_{ij}=
\begin{cases}e_{8j+j_2-4}\otimes e_{8j+j_2-4},\quad i=j-4s;\\
0,\quad\text{otherwise.}\end{cases}$$

If $12s\le j<13s$, then $$z_{ij}=
\begin{cases}w_{8j+7\ra 8(j+1)}\otimes e_{8j+7},\quad i=j-5s;\\
0,\quad\text{otherwise.}\end{cases}$$

(24) $Z^{(24)}_t$ is a $(8s\times 11s)$ matrix with the following nonzero entries
$$z_{6s+(-2)_s,7s+(-2)_s}=w_{8j+6\ra 8j+7}\otimes e_{8j+6};$$
$$z_{6s+(-3)_s,7s+(-3)_s}=w_{8j+6\ra 8j+7}\otimes e_{8j+6};$$
$$z_{7s+(-4)_s,9s+(-4)_s}=w_{8j+7\ra 8(j+1)+4}\otimes e_{8j+7};$$
$$z_{7s+(-3)_s,9s+(-3)_s}=w_{8j+7\ra 8(j+1)+4}\otimes e_{8j+7}.$$

(25) $Z^{(25)}_t$ is a $(8s\times 11s)$ matrix, whose entries $z_{ij}$ have the following form:

If $0\le j<s$, then $$z_{ij}=
\begin{cases}e_{8j}\otimes e_{8j},\quad i=j;\\
0,\quad\text{otherwise.}\end{cases}$$

If $s\le j<2s$, then $$z_{ij}=
\begin{cases}w_{8j+1\ra 8j+4}\otimes e_{8j+1},\quad i=j;\\
0,\quad\text{otherwise.}\end{cases}$$

If $2s\le j<7s$, then $z_{ij}=0$.

If $7s\le j<8s$, then $$z_{ij}=
\begin{cases}-w_{8j+5\ra 8j+6}\otimes e_{8j+5},\quad i=j-2s;\\
0,\quad\text{otherwise.}\end{cases}$$

If $8s\le j<10s$, then $z_{ij}=0$.

If $10s\le j<11s$, then $$z_{ij}=
\begin{cases}-e_{8j+7}\otimes e_{8j+7},\quad i=j-3s;\\
0,\quad\text{otherwise.}\end{cases}$$

(26) $Z^{(26)}_t$ is a $(8s\times 10s)$ matrix with two nonzero entries
$$z_{0,0}=w_{0\ra 4}\otimes e_0\text{ and } z_{0,s}=w_{0\ra 6}\otimes e_0.$$

(27) $Z^{(27)}_t$ is a $(8s\times 9s)$ matrix, whose entries $z_{ij}$ have the following form:

If $0\le j<5s$, then $$z_{ij}=
\begin{cases}w_{8j+j_2\ra 8j+7+j_2}\otimes e_{8j+j_2},\quad i=j;\\
0,\quad\text{otherwise.}\end{cases}$$

If $5s\le j<7s$, then $z_{ij}=0$.

If $7s\le j<8s$, then $$z_{ij}=
\begin{cases}-w_{8j+7\ra 8(j+1)+4}\otimes e_{8j+7},\quad i=j;\\
0,\quad\text{otherwise.}\end{cases}$$

If $8s\le j<9s$, then $z_{ij}=0$.

(28) $Z^{(28)}_t$ is a $(8s\times 8s)$ matrix with unique nonzero entry
$$z_{0,0}=-w_{0\ra 8}\otimes e_0.$$

(29) $Z^{(29)}_t$ is a $(8s\times 8s)$ matrix with unique nonzero entry
$$z_{0,0}=w_{0\ra 8}\otimes e_0.$$

(30) $Z^{(30)}_t$ is a $(8s\times 8s)$ matrix with unique nonzero entry
$$z_{1,1}=w_{1\ra 9}\otimes e_1.$$

(31) $Z^{(31)}_t$ is a $(8s\times 8s)$ matrix with unique nonzero entry
$$z_{2,2}=w_{2\ra 10}\otimes e_2.$$

(32) $Z^{(32)}_t$ is a $(8s\times 8s)$ matrix with unique nonzero entry
$$z_{3,3}=w_{3\ra 11}\otimes e_3.$$

(33) $Z^{(33)}_t$ is a $(8s\times 8s)$ matrix with unique nonzero entry
$$z_{4,4}=w_{4\ra 12}\otimes e_4.$$

(34) $Z^{(34)}_t$ is a $(8s\times 8s)$ matrix with unique nonzero entry
$$z_{5,5}=w_{5\ra 13}\otimes e_5.$$

(35) $Z^{(35)}_t$ is a $(8s\times 8s)$ matrix with unique nonzero entry
$$z_{6,6}=w_{6\ra 14}\otimes e_6.$$

(36) $Z^{(36)}_t$ is a $(8s\times 8s)$ matrix with unique nonzero entry
$$z_{7,7}=w_{7\ra 15}\otimes e_7.$$

\section{Tree class $E_8$: Multiplications in $\HH^*(R)$}

Let $Q_\bullet\rightarrow R$ be the minimal projective bimodule
resolution of the algebra $R$, constructed in Sec.
\ref{sect_res8}. Any $t$-cocycle $f\in\Ker\delta^t$ is lifted
(uniquely up to homotopy) to a chain map of complexes $\{\varphi_i:
Q_{t+i}\rightarrow Q_i\}_{i\ge 0}$. The homomorphism $\varphi_i$ is
called the {\it $i$th translate} of the cocycle $f$ and is
denoted by $\Omega^i(f)$. For cocycles $f_1\in\Ker\delta^{t_1}$ and
$f_2\in\Ker\delta^{t_2}$ we have
\begin{equation}\tag{$**$}\label{mult_formula8}
\cl f_2\cdot \cl f_1=\cl(\Omega^0(f_2)\Omega^{t_2}(f_1)).
\end{equation}

From the descriptions of elements $Z^{(i)}_t$ (given in Sec. \ref{gen8}) and
its $\Omega$-translates (see in ancillary files of this paper) we can find multiplications of all elements, 
except $Z^{(4)}$, $Z^{(8)}$, $Z^{(14)}$, $Z^{(16)}$, $Z^{(22)}$, $Z^{(26)}$ and $Z^{(28)}$,
using the formula \eqref{mult_formula8}. To get the whole picture we should prove the following lemma.

\begin{lem}$\text{ }$

$($a$)$ Let $Z^{(4)}$ be an arbitrary element from generators of the
corresponding type. Then there are elements $Z^{(2)}$ and $Z^{(3)}$
such as $Z^{(4)}=Z^{(2)}Z^{(3)}$.

$($b$)$ Let $Z^{(8)}$ be an arbitrary element from generators of
the corresponding type. Then there are elements $Z^{(2)}$ and
$Z^{(7)}$ such as $Z^{(8)}=Z^{(2)}Z^{(7)}$.

$($c$)$ Let $Z^{(14)}$ be an arbitrary element from generators of
the corresponding type. Then there are elements $Z^{(2)}$ and
$Z^{(13)}$ such as $Z^{(14)}=Z^{(2)}Z^{(13)}$.

$($d$)$ Let $Z^{(16)}$ be an arbitrary element from generators of
the corresponding type. Then there are elements $Z^{(2)}$ and
$Z^{(15)}$ such as $Z^{(16)}=Z^{(2)}Z^{(15)}$.

$($e$)$ Let $Z^{(22)}$ be an arbitrary element from generators of
the corresponding type. Then there are elements $Z^{(2)}$ and
$Z^{(21)}$ such as $Z^{(22)}=Z^{(2)}Z^{(21)}$.

$($f$)$ Let $Z^{(26)}$ be an arbitrary element from generators of
the corresponding type. Then there are elements $Z^{(2)}$ and
$Z^{(25)}$ such as $Z^{(26)}=Z^{(2)}Z^{(25)}$.

$($g$)$ Let $Z^{(28)}$ be an arbitrary element from generators of
the corresponding type. Then there are elements $Z^{(2)}$ and
$Z^{(27)}$ such as $Z^{(28)}=Z^{(2)}Z^{(27)}$.

\end{lem}

\begin{proof}
The degree 1 has type 2, for all $s$. It only remains to use the
relations for type (2).
\end{proof}



\begin{thebibliography}{99}

\bibitem{Riedt}
C. Riedtmann, {\it Algebren, Darstellungsk\"ocher, \"Uberlagerungen
und zur\"uck}. --- Comment. Math. Helv., 1980, v. 55, 199--224.

\bibitem{Erd}
K. Erdmann, T. Holm, {\it Twisted bimodules and Hochschild
cohomology for self-injective algebras of class $A_n$}. --- Forum
Math., 1999, v. 11, 177--201.

\bibitem{Gen&Ka}
A. I. Generalov, M. A. Kachalova, {\it Bimodule Resolution of
M\"obius Algebras}. --- Zap. Nauchn. Semin. POMI, {\bf 321} (2005),
36--66.

\bibitem{Ka}
M. A. Kachalova, {\it Hochschild cohomology of M\"obius Algebras}.
--- Zap. Nauchn. Semin. POMI, {\bf 330} (2006), 173--200.

\bibitem{Pu}
M. A. Pustovykh, {\it Hochschild cohomology ring of M\"obius
algebras}. --- Zap. Nauchn. Semin. POMI, {\bf 388} (2011), 210--246.

\bibitem{Volkov1}
Yu. V. Volkov, A. I. Generalov, {\it Hochschild cohomology for
self-injective algebras of tree class $D_n$}. I.
--- Zap. Nauchn. Semin. POMI, {\bf 343} (2007), 121--182.

\bibitem{Volkov2}
Yu. V. Volkov, {\it Hochschild cohomology for self-injective
algebras of tree class $D_n$}. II. --- Zap. Nauchn. Semin. POMI,
{\bf 365} (2009), 63--121.

\bibitem{Volkov3}
Yu. V. Volkov, A. I. Generalov, {\it Hochschild cohomology for
self-injective algebras of tree class $D_n$}. III. --- Zap. Nauchn.
Semin. POMI, {\bf 386} (2011), 100--128.

\bibitem{Volkov4}
Yu. V. Volkov, {\it Hochschild cohomology for nonstandard
self-injective algebras of tree class $D_n$}. --- Zap. Nauchn.
Semin. POMI, {\bf 388} (2011), 48--99.

\bibitem{Volkov5}
Yu. V. Volkov, {\it Hochschild cohomology for self-injective
algebras of tree class $D_n$}. IV. --- Zap. Nauchn. Semin. POMI,
{\bf 388} (2011), 100--118.

\bibitem{Volkov6}
Yu. V. Volkov, {\it Hochschild cohomology for self-injective
algebras of tree class $D_n$}. V. --- Zap. Nauchn. Semin. POMI, {\bf
394} (2011), 140--173.

\bibitem{Pu2}
M. A. Pustovykh, {\it Hochschild cohomology ring for self-injective algebras of tree class $E_6$}. --- 
Zap. Nauchn. Semin. POMI, {\bf 423} (2014), 205--243.

\bibitem{Ka2}
M. A. Kachalova, {\it Hochschild cohomology ring for self-injective algebras of tree class $E_6$. II}.
--- Zap. Nauchn. Semin. POMI, {\bf 478} (2019), 128--171.

\bibitem{Ka3}
M. A. Kachalova, {\it Hochschild cohomology ring for self-injective algebras of tree class $E_7$}.
--- Zap. Nauchn. Semin. POMI, {\bf 484} (2019), 86--114.

\bibitem{Ha}
D. Happel, {\it Hochschild cohomology of finite-dimensional
algebras}. --- Lect. Notes Math., 1989, 1404, 108--126.

\bibitem{VGI}
Yu. V. Volkov, A. I. Generalov, S. O. Ivanov, {\it On construction
of bimodule resolutions with the help of Happel's lemma}. --- Zap.
Nauchn. Semin. POMI, {\bf 375} (2010), 61--70.

\end{thebibliography}
\end{document}